\documentclass[11pt,a4paper]{article}

\usepackage[naustrian, english]{babel} 
\usepackage[T1]{fontenc}
\usepackage{lmodern}
\usepackage{textcomp}
\usepackage{amssymb} 
\usepackage{amsthm}
\usepackage{mathtools}
\usepackage[pdftex]{graphicx}
\usepackage[export]{adjustbox}
\usepackage{color}
\usepackage[arrow, matrix, curve]{xy}
\usepackage{lmodern}
\usepackage[utf8]{inputenc}
\usepackage{CJKutf8}
\usepackage{url}
\usepackage{hyperref}
\usepackage{float}
\usepackage{chngcntr}
\usepackage{listings}
\usepackage{appendix}
\usepackage{pdfpages}
\usepackage{algorithm}
\usepackage{algorithmic}
\usepackage{geometry}
\usepackage{bbm}
\usepackage{bbold}
\usepackage{graphicx}
\usepackage{caption}
\usepackage{subcaption}
\usepackage{xcolor}
\usepackage{xargs} 
\usepackage[colorinlistoftodos,textsize=small]{todonotes}
\usepackage{relsize}
\usepackage{tikz}
\usetikzlibrary{decorations.pathmorphing, decorations.pathreplacing, decorations.shapes, arrows.meta}

\definecolor{codegreen}{rgb}{0,0.6,0}
\definecolor{codegray}{rgb}{0.5,0.5,0.5}
\definecolor{codepurple}{rgb}{0.58,0,0.82}

\lstdefinestyle{mystyle}{
	basicstyle=\tiny,
	commentstyle=\color{codegreen},
	keywordstyle=\color{blue},
	numberstyle=\tiny\color{codegray},
	stringstyle=\color{codepurple},
	breakatwhitespace=false,         
	breaklines=true,                 
	captionpos=b,                    
	keepspaces=true,                 
	numbers=left,                    
	numbersep=5pt,                  
	showspaces=false,                
	showstringspaces=false,
	showtabs=false,                  
	tabsize=2,
}

\newcommand{\R}{\mathbb{R}}

\newcommand{\pa}{\partial}

\newcommand{\ve}{\varepsilon}

\newcommand{\vp}{\varphi}

\newcommand{\md}{\,\mathrm{d}}

\newcommand{\supp}{\operatorname{supp}}

\newcommand{\m}{\mathcal{M}}

\newtheorem{theorem}{Theorem}
\newtheorem{lemma}[theorem]{Lemma}

\theoremstyle{definition}
  
\newtheorem{rem}{Remark}

\newcommandx{\unsure}[2][1=]{\todo[linecolor=red,backgroundcolor=red!25,bordercolor=red,#1]{#2}}
\newcommandx{\change}[2][1=]{\todo[linecolor=blue,backgroundcolor=blue!25,bordercolor=blue,#1]{#2}}
\newcommandx{\info}[2][1=]{\todo[linecolor=green,backgroundcolor=green!25,bordercolor=green,#1]{#2}}
\newcommandx{\improvement}[2][1=]{\todo[linecolor=yellow,backgroundcolor=yellow!25,bordercolor=yellow,#1]{#2}}

\newcommandx{\biblio}[2][1=]{\todo[linecolor=blue,backgroundcolor=magenta!25,bordercolor=blue,#1]{#2}}
\newcommandx{\laura}[2][1=]{\todo[linecolor=violet,backgroundcolor=violet!25,bordercolor=violet,#1]{#2}}
\newcommandx{\michi}[2][1=]{\todo[linecolor=purple,backgroundcolor=red!40,bordercolor=purple,#1]{#2}}

\begin{document}
	\title{Structured Model Conserving Biomass for the Size-spectrum Evolution in Aquatic Ecosystems}
	\author{L. Kanzler\thanks{CEREMADE, Université Paris Dauphine, Place du Maréchal de Lattre de
Tassigny, F-75775 Paris Cedex 16, France.\tt{laura.kanzler@dauphine.psl.eu}} \and B. Perthame\thanks{Sorbonne Université, CNRS, Université de Paris Cit\'e, Inria, Laboratoire Jacques-Louis Lions, 75005 Paris, France. \tt{benoit.Perthame@sorbonne-universite.fr}} \and B. Sarels\thanks{Sorbonne Université,  CNRS, Université de Paris Cit\'e, Laboratoire Jacques-Louis Lions, 75005 Paris, France. \tt{benoit.sarels@sorbonne-universite.fr}}}
	\date{\vspace{-5ex}}

\maketitle

\begin{abstract}
Mathematical modelling of the evolution of the \emph{size-spectrum dynamics in aquatic ecosystems} was discovered to be a powerful tool to have a deeper insight into impacts of human- and environmental driven changes on the marine ecosystem. In this article we propose to investigate such dynamics by formulating and investigating a suitable model. The underlying process for these dynamics is given by predation events, causing both growth and death of individuals, while keeping the total biomass within the ecosystem constant. The main governing equation investigated is deterministic and non-local of quadratic type, coming from binary interactions. Predation is assumed to strongly depend on the ratio between a predator and its prey, which is distributed around a preferred feeding preference value. Existence of solutions is shown in dependence of the choice of the feeding preference function as well as the choice of the search exponent, a constant influencing the average volume in water an individual has to search until it finds prey. The equation admits a trivial steady state representing a died out ecosystem, as well as - depending on the parameter-regime - steady states with gaps in the size spectrum, giving evidence to the well known \emph{cascade effect}. The question of stability of these equilibria is considered, showing convergence to the trivial steady state in a certain range of parameters. These analytical observations are underlined by numerical simulations, with additionally exhibiting convergence to the non-trivial equilibrium for specific ranges of parameters. 
\end{abstract}


\medskip
	
		
\medskip

\section{Introduction}\label{s:intro}

Understanding the size-spectrum dynamics in aquatic ecosystems is an active field of ecology. For that purpose, several types of models have been proposed to describe the underlying processes based on predation events. These lead to equations depicting growth and decay of interacting populations structured by the average size of individuals. We modify one of the most popular and recent models \cite{DDL} to include the production of individuals of ’small’ sizes so as to keep the total biomass constant. The resulting model, which we aim to investigate within this work, describes the time-evolution of the distribution function $f(w,t)$, encoding the abundance of individuals with body size/weight $w \in \R_+ = (0,\infty)$ at time $t \geq 0$ in the ecosystem. The model is of \emph{kinetic collisional-type}, spatially homogeneous, structurally similar to \emph{kinetic equations for coalescing particles} \cite{S} and the underlying individual-based dynamics are given by binary predation events modelled by the following \emph{process}: A predator feeds on a prey, resulting instantaneously in growth of the predator as well as occurrence of a certain amount of small individuals ('offspring'). Main matter of investigation will be an equation of the form 
\begin{equation}\label{eq:e:BEpdelta}
	\begin{split}
		\pa_t f(w) &= \int_0^{\frac{w}{K}} k(w-Kw',w') f(w-Kw') f(w') \, \md w' + \frac{1-K}{K'^2} \int_0^{\infty} k\left(w',\frac{w}{K'}\right) f\left(\frac{w}{K'}\right) f(w') \,\md w' \\
		&-\int_0^{\infty} \left(k(w,w')+k(w',w)\right) f(w')f(w) \, \md w',
	\end{split}
\end{equation}
being a special case of a class of such models, more precisely introduced in Section \ref{s:model}. The frequency of predation events described by \eqref{eq:e:BEpdelta} is quantified by the \emph{feeding kernel} $k(\cdot,\cdot)$, where the first variable stands for the predator weight, while the second always holds the prey weight. Following \cite{BR}, typical feeding kernels strongly depend on the ratio between the weights of predator and prey and are of the form 
\begin{align}\label{g:kernel}
	k(w,w')=Aw^{\alpha}s\left(\frac{w}{w'}\right).
\end{align}
The quantity $Aw^{\alpha}$ describes the volume searched per unit time by an individual with size $w$, which is modelled as an \emph{allometric function} of the animal weight \cite{B,K,W}. Hence, the constant $A>0$ denotes the search volume per unit mass$^{-\alpha}$ and per unit time, while $\alpha>0$ encodes the \emph{search exponent}. The first can be seen as a time-scaling with little impact on the qualitative dynamics when of order 1, which makes it reasonable to set $A=1$. The function $s: \R_+ \to \R_+$ is the so-called \emph{feeding preference} function, encoding the preferred prey size proportional to the predator body weight. After such a predation event the predating individual assimilates a certain, usually very small, fraction of its prey weight, given by the \emph{assimilation constant} $K>0$, in literature also known as \emph{Lindeman efficiency} \cite{B}. At the same time a certain amount of 'offspring' is produced, i.e. individuals having a fraction $K'$, with $K' \ll 1$, of the prey weight. The main novelty of our proposed equation compared to the model introduced in \cite{DDL} is the second term on the right-hand-side of \eqref{eq:e:BEpdelta}, being responsible for the gain of very small individuals, as side product of a predation, which causes an inflow close to zero. Moreover, the dynamics of this model will not cause the ecosystem to die out due to a decreasing pool of possible prey. Having in mind the modelling of rather small, closed-up ecosystems, this further goes along with the biologically reasonable assumption of conservation of total biomass of the whole system. 

\vspace{1cm}

Trophic interactions between animals in the ocean have been a matter of interest since the '60s with the first size-measurements of individuals taken and investigated by Paloheimo et al. \cite{PD}, Hairston et al. \cite{HSS} and Sheldon et al. in \cite{SSP,SPS}. The individual body size was discovered to be the ’master trait’ in food webs of animals \cite{E}, giving rise to emergent distributions of biomass, abundance and production of organisms. In fact, in marine and freshwater ecosystems it was conjectured \cite{BD,SPS,SSP} that by treating individuals as particles with states given by their size, equal intervals of \emph{biomass} (i.e. body weight $\times$ abundance) 
\begin{align*}
	M(w) = w f(w)
\end{align*}
in logarithmic intervals of the organism body weight are observed to approximately contain equal amounts of biomass per unit volume. This phenomenon is today known as the \emph{Sheldon conjecture} \cite{SPS} and is equivalent to the biomass function $M$ having slope -1 in logarithmic scales. Mathematically speaking, for biomass in weight range $[w_1,w_2]$, the following should hold true
\begin{align*}
	\ln{\left(\frac{w_2}{w_1}\right)} = \int_{w_1}^{w_2} M(w) \, \md w,
\end{align*}
or equivalently by denoting $x := \ln(w)$,
\begin{align*}
	\ln{(M(e^x))}=- x.
\end{align*}
Hence, biomass density decreases approximately as the inverse of body mass. While this \emph{power-law relation} seems to be a good approximation for large ecosystems including a huge range of trophic levels in a wide geographic area \cite{SPS}, different studies dedicated to smaller scaled ecosystems very often reveal the occurrence of \emph{dome-patterns} \cite{Eetal, HSS, RGK}. These phenomena, also known as \emph{cascade effects}, describes the suppression of specific trophic levels in the ecosystem as result of an indirect influence from one trophic level to the second next lower or higher. The graph of the corresponding size spectrum function will then have intervals in which it is zero, with each such gap indicating the absence of specific trophic levels. 

Based on this important observations from the '60s and '70s, size-based ecosystem modelling was discovered as a powerful tool to have a deeper insight into impacts of human- and environment-driven changes on the marine ecosystem, giving rise to a variety of models to capture this phenomenon. While these proposed mathematical models describing mass spectra at large scales in aquatic ecosystems seem very simple, in fact they pose deep mathematical questions, which reflects the high computational cost to solve them numerically and the complex patterns of solutions which may confirm hypotheses by ecologists.

Size-spectrum models (SSMs) have been developed starting with Silvert and Platt \cite{SP, SP2}, followed by Benoît and Rochet \cite{BR}, Andersen and Beyer \cite{AB, AB2}, Capitán and Delius \cite{CD}, Datta et al. \cite{DDL, DDLP} and Cuesta, Delius and Law \cite{CDL}. A common feature is that the body weights change due to interactions between organisms at different sizes. Individuals grow by feeding on and killing smaller organisms, thus connecting the two opposing effects of predation: death of the prey, and body growth of the predator.  A common feature of these models is the allometric scaling of the rates of the different processes. This scales back to observations in the '30s by Huxley \cite{H}, who stated that most size-related variations of individual characteristics can be expressed as power-law functions of the body mass. Especially, shown by Kleiber \cite{K}, the metabolic rate follows such a power-law with exponent 3/4, see also \cite{B}. For a broad overview of size-spectrum models developed so far we refer the reader to \cite{BHETR}, giving an historical scope over the evolution in this field including a wide range of existing models as well as their connections. On the mathematical level, these dynamics are expressed in terms of partial differential equations (PDEs), whose structures are similar to several equations recently used in mathematical biology \cite{P}. In \cite{SP,SP2} the authors propose such a \emph{McKendrick-von Foerster equation} \cite{M, F} as model with growth and mortality to be functions of body mass, coupled by the aforementioned predation events. This was specified in later works, in which the predation was restricted to organisms of smaller size \cite{CD} before in \cite{BR} the authors introduced the feeding kernel giving a more precise description of the feeding behaviour within the ecosystem. In more recent models \cite{DDL, DDLP} the authors aim to overcome the discrepancy that the evolution of an organism body weight does not follow the same rules as ageing of individuals. Indeed, growth in size is heavily influenced by interaction with other individuals (i.e. by predation), while ageing happens uniformly in the population without necessity of interactions. The authors hence proposed a \emph{jump-growth} model encoding these predation events within integral terms of quadratic order, where the aforementioned age-structured \emph{McKendrick–von Foerster equation} can be recovered as a first order approximation in a specific parameter-regime involving a very small biomass assimilation constant. Numerical experiments in \cite{BR} as well as in \cite{DDL} were performed. In both cases the dynamics within individuals below (e.g. phytoplankton and lower) and above (e.g. top carnivores) certain weight thresholds are assumed to be governed by simpler dynamics, as, for example, relaxation equations with an equilibrium reservoir of organisms in this trophic level. This is especially necessary to overcome the lack of inflow of mass, since predation just results in growth of one species, while a huge amount of the prey mass is lost. Simulations showed the occurrence of travelling waves and oscillatory solutions, indicating already the high instability and hence sensitivity of the power-law steady state to perturbations. Moreover, in \cite{DDLP} the stability of the power-law equilibrium was analytically investigated using tools from spectral analysis in the case of the \emph{McKendrick–von Foerster equation}, the \emph{jump-growth} model introduced in \cite{DDL} as well as the McKendrick–von Foerster type equations with an additional diffusion-term, giving evidence to its instability in almost all reasonable parameter-regimes. Indeed, especially on smaller scaled ecosystems, rather \emph{dome patterns} in the size-spectrum function can be observed \cite{RGK}, giving evidence to the so-called \emph{cascade-effect} \cite{Eetal, HSS} within ecosystems. 

Therefore, we propose a model with the biologically very reasonable assumption of conservation of the total biomass of the ecosystem under investigation. This is realised by a gain of small individuals as byproduct of each predation event. This conserved quantity, mathematically interpreted as the first moment of the distribution function $f$, gives additional structure needed for rigorous analytical investigation of the proposed equation. Moreover we are able to capture the widely observed \emph{cascade effect} by carefully chosen feeding interaction functions, indicating its occurrence or absence of specific trophic levels in the ecosystem. 

The underlying microscopic dynamics of \eqref{eq:e:BEpdelta} are given by binary, instantaneous interactions, which bears some resemblance with a variety of binary collisional models in the field of \emph{kinetic theory}, starting with the \emph{Boltzmann equation} as the most famous example \cite{CIP, V}. Predation, the type of binary interaction considered in our model, can be interpreted as coagulation with loss of mass in the context of coalescing particles. Structural similarity to the \emph{Smoluchowski coagulation equation} \cite{S} as well as to further \emph{coagulation equations}, matter of recent studies \cite{BKBSHSS, CR, ELM, EP, GL1, GL2, NV, Sr}, follows naturally. We shall make use of some techniques arising from these fields of research when deriving analytical properties of the models: in particular control over characteristic moments, fixed-point combined with compactness arguments as well as entropy-and entropy dissipation techniques, especially characteristic for such structured equations in mathematical biology and physics \cite{P, RVG}. 

The article is organised as follows: In Section \ref{s:model} we introduce and motivate our general size-spectrum model for aquatic ecosystems, while embedding in existing literature, showing similarities and crucial differences. Moreover, for a certain case, main focus of our analysis, formal properties regarding control of moments are derived. The question of existence and uniqueness is investigated in the subsequent Section \ref{s:existence}. Section \ref{s:longtime} is dedicated to the study of existence and admissibility of steady states, as well as convergence towards them regarding the parameter-regime. Finally in Section \ref{s:num} we give evidence to these analytical observations by providing numerical simulations underlying results and conjectures stated in the previous Section \ref{s:longtime}.

\section{A Jump-Growth Model for Size-spectrum in Aquatic Ecosystems}\label{s:model}

We introduce a new class of models for aquatic ecosystems, which contains production of small individuals ('offspring') and, hence, conserves the total biomass. Following \cite{DDL} we assume that the underlying individual-based mechanism is given by a \emph{Markov process}, describing binary predation events of single organisms in the ecosystem, which are modelled in the following way: Predation is assumed to happen in an instant, followed by jumps in the size of the predator and, different to the model in \cite{DDL}, production of a certain amount of very small sized organisms, for simplicity called 'offspring' throughout this article, which are nourishing the ecosystem at the smallest size class. This modelling choice aims to capture the phenomenon observed in ecosystems, especially in rather closed-up and, hence, self-contained environments, in which all the available materials are re-utilised through the establishment of continuous cycles. To be more precise, a predator with body-weight $w \in \R_+$ feeds on a prey with weight $w’ \in \R_+$. The predator is able to assimilate a fraction $Kw'$ of the body mass of its prey. This implies that the predator post-feeding state is given by $w+Kw’$. The assimilation constant $K$ is assumed to be small, i.e. $0<K<1$, inspired by insights regarding metabolic theory of ecology \cite{AB, AB2, B}, widely used by ecologists. Furthermore, we assume that the body mass of the prey, which cannot be assimilated by the predator, will produce a certain amount $P>0$ of 'offspring'. Therefore, we introduce the waste-to-nutrient density $p(w,w')$, which encodes the probability that within a predation event, where an individual with weight $w'$ is eaten, 'offspring' of size $w$ is produced. We assume that $p(w,\cdot)$ is a probability density function, i.e.
$$
\int_0^{\infty} p(w,w') \, \md w' =1, \quad \forall w \in [0,\infty).
$$
A predation event is assumed to happen with a certain rate given by the \emph{feeding kernel} $k(\cdot,\cdot)$ depending on the prey and predator weights \eqref{g:kernel}, naturally asymmetric in its two arguments.

The associated evolution equation for the distribution function $f(w,t), w \in \R_+, t\geq0$ of the individuals in the ecosystem, is of collisional-type, spatially homogeneous and reads as
\begin{equation}\label{e:BEp}
	\begin{split}
		\pa_t f(w) =& \int_0^{\frac{w}{K}} k(w-Kw',w') f(w-Kw') f(w') \, \md w' \\
		+& P \int_0^{\infty} \int_0^{\infty} p(w,w'')k(w',w'') f(w'') f(w')\, \md w'' \,\md w' \\
		-&\int_0^{\infty} \left(k(w,w')+k(w',w)\right) f(w')f(w) \, \md w'.
	\end{split}
\end{equation}
The two loss terms on the right-hand-side encode that within a predation event in the ecosystem two individuals are lost and enter the dynamics again with post-predation states, given by the two gain terms. The first gain term encodes the addition of an individual which grew after eating its prey. The second gain term on the other hand describes the gain of 'offspring' as a side-product induced by the remaining body mass, which could not be assimilated by the predating organism. We want to point out that the appearance of the two loss terms comes from the asymmetry of the feeding kernel. Equation \eqref{e:BEp} bears some strong structural similarity to the \emph{collision-induced breakage equation}, which is used to model formation of raindrops as well as planetesimals. Several important contributions to the understanding to such types of equations can be found over the last four decades from a physical point of view starting with \cite{Sr}, for a model describing the size-distribution of raindrops, followed by \cite{CR} for a general model for irreversible fragmentation and \cite{BKBSHSS} with a study of size distribution of particles in Saturn’s rings, which is also governed by coagulation and fragmentation processes. Moreover, see \cite{EP, GL1, GL2} for more recent theoretical findings regarding this type of equation.

The \emph{weak formulation} of \eqref{e:BEp} obtained by multiplication by a suitable test-function $\vp: \R_+ \to \R_+$ before integrating over the state space $\R_+$ is given by 
\begin{equation}\label{e:weakBEp}
	\begin{split}
		&\frac{\md}{\md t} \int_0^{\infty} f(w) \vp(w) \, \md w \\
		=& \int_0^{\infty} \int_0^{\infty} \left(\vp\left(w+Kw'\right)+ P \int_0^{\infty} \vp(w'') p(w'',w') \, \md w'' - \vp(w) - \vp(w')\right)k(w,w') f(w) f(w') \, \md w' \, \md w.
	\end{split}
\end{equation}
For having a rigorous justification for the derivation of the equation for evolution of the distribution function for the microscopic dynamics, we again refer the reader to \cite{DDL}, where this topic was discussed for their model with very similar structure. Moreover, we would like to point out more precisely the structural similarity to the \emph{collision-induced breakage equation} investigated in \cite{GL1}. The differences lie, on the one hand, in the asymmetry of the feeding kernel $k(\cdot,\cdot)$, which is assumed to be symmetric in the collision-induced breakage setting. On the other hand, the equation in \cite{GL1} features the possibility of aggregation as well as breakup induced by a collision of two particles, while the model \eqref{e:BEp} just allows for the latter, since after each predation an amount of 'offspring' is produced. Further, different to the equation in \cite{GL1}, the function describing the daughter distribution after the breakup is symmetric, while this would not be reasonable in the context of predation behaviour.

We aim to model an ecosystem where the \emph{total biomass} of the ecosystem 
\begin{align}\label{d:biomass}
	\m:= \int_0^\infty w f(w,t) \, \md w
\end{align}
is formally conserved. This asks for conservation of biomass within each predation event, which results in the restriction to choose $P$ and $p$, fulfilling the relation
\begin{align*}
	P \int_0^{\infty} w'' p(w'',w') \, \md w'' = (1-K)w', \quad \forall w' \in [0,\infty),
\end{align*}
which can be seen by the weak formulation \eqref{e:weakBEp} with the choice $\vp(w)=w$. It is, however, well-investigated that for equations similar to the Smoluchowski coagulation equation \cite{S} a phenomenon known as \emph{gelation} can occur, which describes the loss of mass at finite time, associated with the blow-up of the first moment, indicating formation of an infinite cluster of particles. This will be further discussed in Section \ref{ss:properties}.

\begin{rem}
The quantity 
$$
m_p(w'):=\int_0^{\infty} w'' p(w'',w') \, \md w''
$$
describes the mean 'offspring'-size, which appears after an individual of weight $w'$ is eaten. We make the reasonable assumption that $\frac{m_p(w')}{w'} = \frac{1-K}{P}w' \ll 1$, hence $P \gg 1$ has to hold for all $w' \in (0,\infty)$. 
\end{rem}

\subsection{Deterministic Jump-Growth Model with Offspring-production}\label{ss:dmodel}

A very specific choice for the waste-to-nutrient density $p(\cdot,\cdot)$ would be 
$$
p(w,w') = \delta_{\frac{w}{K'}}(w') = K' \delta_{K'w'}(w),
$$
for positive very small constant $0<  K' \ll 1$, which means that after each predation event the ecosystem is nourished with particles, whose size is given by a very small fraction of the size of the consumed prey. Hence, the stochasticity in the dynamics on the level of the distribution function is completely lost, leaving a purely deterministic model on which we will focus in the following. Moreover, in that case we are able to determine the constant $P$ explicitly, namely $P=\frac{1-K}{K'^2}$. A visualisation of the underlying \emph{jump-growth process} can be found in Figure \ref{f:jump_growth}.
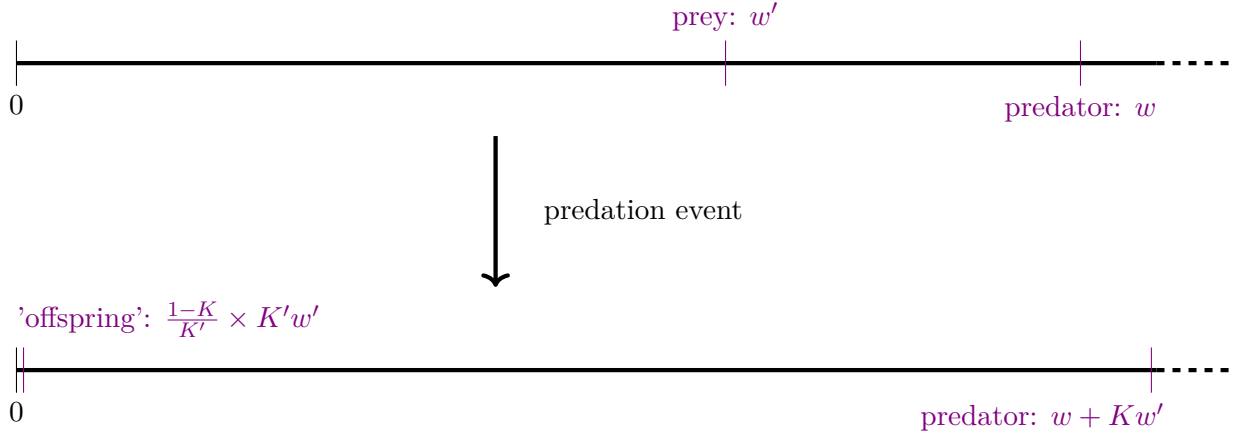
\begin{figure}
	\centering
	\begin{tikzpicture}
		\draw[ultra thick, -] (0,0) -- (15,0);
		\draw[ultra thick, dashed] (15,0) -- (16,0);
		\draw[-] (0,0.3) -- (0,-0.3) node[anchor=south, below]{0};
		\draw[violet, -] (9.33,0.3)node[anchor=north, above]{prey: $w'$} -- (9.33,-0.3);
		\draw[violet, -] (14,0.3) -- (14,-0.3) node[anchor=south, below]{predator: $w$};
	\end{tikzpicture}
	\begin{tikzpicture}
		\draw[ultra thick, - >] (9,1) to (9,-1);
		\node at (9.5,0) [anchor=east, right]{predation event};
	\end{tikzpicture}
	\begin{tikzpicture}
		\draw[ultra thick, -] (0,0) -- (15,0);
		\draw[ultra thick, dashed] (15,0) -- (16,0);
		\draw[-] (0,0.3) -- (0,-0.3) node[anchor=south, below]{0};
		\draw[violet, -] (0.0933,0.3) -- (0.0933,-0.3); 
		\node at (2,0.3) [anchor=north, above, violet]{'offspring': $\frac{1-K}{K'} \times K'w'$};
		\draw[violet, ] (14.933,0.3) -- (14.933,-0.3); 
		\node at (13.5,-0.3) [anchor=south, below, violet]{predator: $w+Kw'$};
	\end{tikzpicture}
	\caption{Visualisation of the underlying individual based interaction rules, governing the dynamics encoded by equation \eqref{e:PJG}. A predator with weight $w$ grows and jumps to state $w+Kw'$, where $w'$ is the prey weight, while producing an amount $\frac{1-K}{K'}$ of 'offspring' of size $K'w'$ as a side-product of the biomass, which could not be assimilated by the predator.} 
	\label{f:jump_growth}
\end{figure}

Under these considerations and additionally assuming that the feeding kernel is of the form \eqref{g:kernel} the model becomes
\begin{align}\label{e:PJG}
	\pa_t f(w) &= Q(f,f)=G_1(f,f)+G_2(f,f) - L_1(f,f)-L_2(f,f)\notag  \\
	:=&\int_0^{\frac{w}{K}} \left(w-Kw'\right)^{\alpha} s\left(\frac{w-Kw'}{w'}\right) f(w-Kw') f(w') \, \md w' \notag \\
	\phantom{=}+& \frac{1-K}{K'^{2}} \int_0^{\infty} w'^{\alpha} s\left(\frac{w'K'}{w}\right)f\left(\frac{w}{K'}\right) f(w') \,\md w' \\
	\phantom{=}-& \int_0^{\infty} \left(w^{\alpha} s\left(\frac{w}{w'}\right)+w'^{\alpha} s\left(\frac{w'}{w}\right) \right) f(w')f(w) \, \md w', \notag
\end{align}
which equipped with initial conditions
\begin{align}\label{IC_PJG}
	f(w,0)=f_0(w), \quad w \in \R_+
\end{align}
is the matter of investigation in the remainder of this article. 

Inspired by the choice of  $s(\cdot)$ we perform the coordinate change $w' \to r:=w_{predator}/w_{prey}$ in every term of \eqref{e:PJG} to obtain 
\begin{align}\label{e:PJG_r}
	\pa_t f(w) &= w^{\alpha +1}\int_0^\infty s(r) \left[r^{\alpha}(r+K)^{-\alpha-2} f\left(\frac{wr}{r+K}\right)f\left(\frac{w}{r+K}\right) \right. \\
	&\left.+ r^{\alpha}(1-K)K'^{-3-\alpha} f\left(\frac{w}{K'}\right)f\left(\frac{wr}{K'}\right) -r^{-2} f(w) f\left(\frac{w}{r}\right) - r^{\alpha} f(w)f(rw) \right] \, \md r \notag
\end{align}
which allows a clearer vision of the influence of the feeding ratio $r \in [0,\infty)$ on the dynamics.

\paragraph{The Choice of the feeding preference function:}

Since, in this model approach, the evolution of the aquatic system is only driven by predation with the body size being the decisive trait, the choice of the \emph{feeding preference function} $s:\R_+\to\R_+$ has essential influence on the dynamics in the ecosystem. In this article we consider two possibilities. On the one hand, we assume that the feeding ratio $r$ of predator and prey is drawn by a Gaußian distribution
\begin{align}\label{d:sgauss}
	s(r):=\frac{1}{\sigma\sqrt{2\pi}}\exp\left({-\frac{(r-B)^2}{2\sigma^2}}\right),
\end{align}
already used in a similar manner in existing models as in \cite{BR,DDL,DDLP}. On the other hand, we work with 
\begin{align}\label{d:scomp}
	s(r):= \frac{1}{\sigma^2}\exp{\left(\frac{-\sigma^2}{\sigma^2-(r-B)^2}\right)} \mathbb{1}_{[B-\sigma,B+\sigma]}(r).
\end{align}
Although of similar shape, as depicted in Figure \ref{F:s}, the crucial difference compared to \eqref{d:sgauss} is its compact support. Throughout the analysis-part of this article we focus on choice \eqref{d:scomp} out of convenience. Doing so, we will not neglect to point out the main analytical differences and, most importantly, the biological significance of these modelling choices, supported by numerical simulations in Section \ref{s:num}.

In both cases, the constant $B>0$ describes the \emph{preferred feeding ratio} between predator and its prey, while the parameter $\sigma>0$ denotes the \emph{variance from this feeding preference} $B$. We specially observe that for both choices \eqref{d:sgauss} and \eqref{d:scomp} $s \in C^{\infty}([0,\infty))$ hold and that both functions are bounded from above respectively by $\frac{1}{\sigma\sqrt{2\pi}}$ and $\frac{1}{e\sigma^2}$, their respective maximal values when $r$ reaches the preferred feeding ratio $B$. 

The main difference being that in the first modelling choice \eqref{d:sgauss} all feeding ratios between predator and prey are possible, although most of them, depending on the width $\sigma$ of the Gaußian, with quite low probability. Especially, these dynamics also include the biologically not applicable situations of predating organisms with arbitrary small size as well as preys of arbitrary big size. The second feeding preference function \eqref{d:scomp} just admits a variance $\sigma$ around the preferred feeding ratio, while predation interactions between individuals with weights such that $w_{predator}/w_{prey} \notin  [B-\sigma,B+\sigma]$ are not admitted. As it can be seen in Section \ref{s:longtime}, supported by simulations in Section \ref{s:num}, this will be the crucial property, which admits non-zero quasi-stationary solutions of \eqref{e:PJG}-\eqref{IC_PJG} with gaps in the size-spectrum, while this cannot be the case for the Gaußian feeding preference function. 
\begin{figure}[H]
	\centering
	\begin{subfigure}{0.45\textwidth}
		\includegraphics[width=8.5cm]{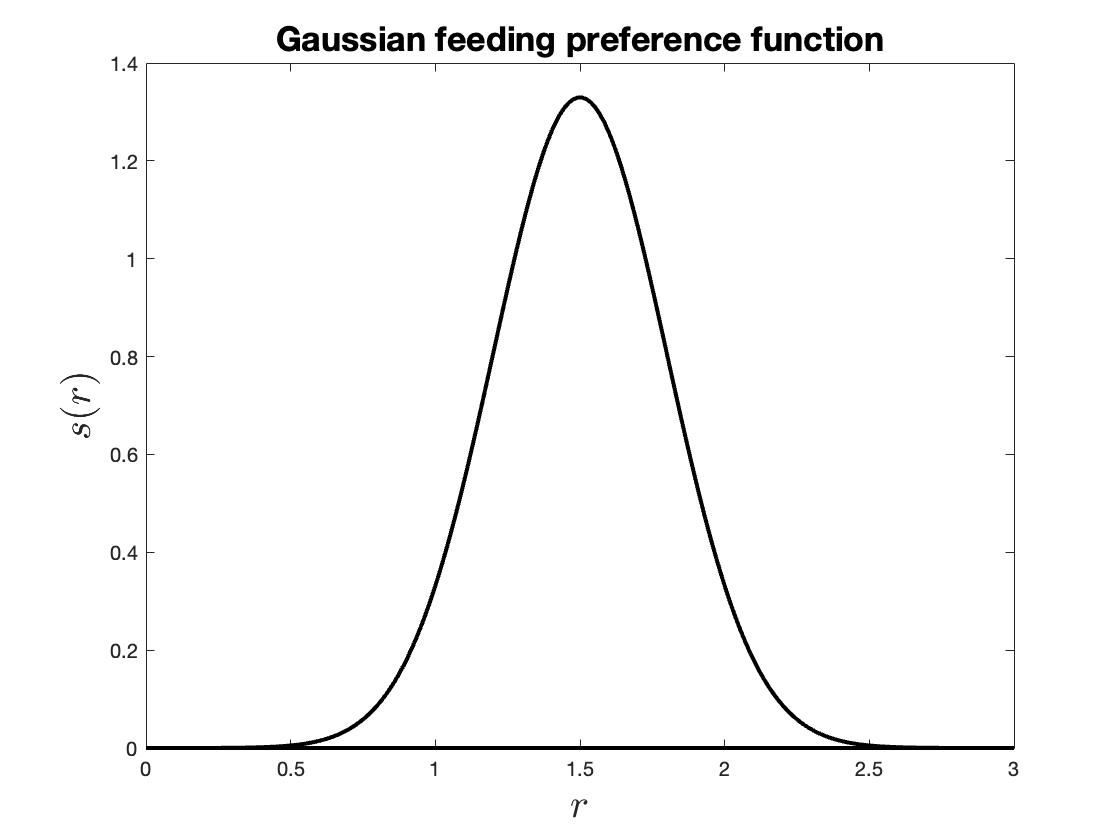}
	\end{subfigure}
	\hfill
	\begin{subfigure}{0.45\textwidth}
		\includegraphics[width=8.5cm]{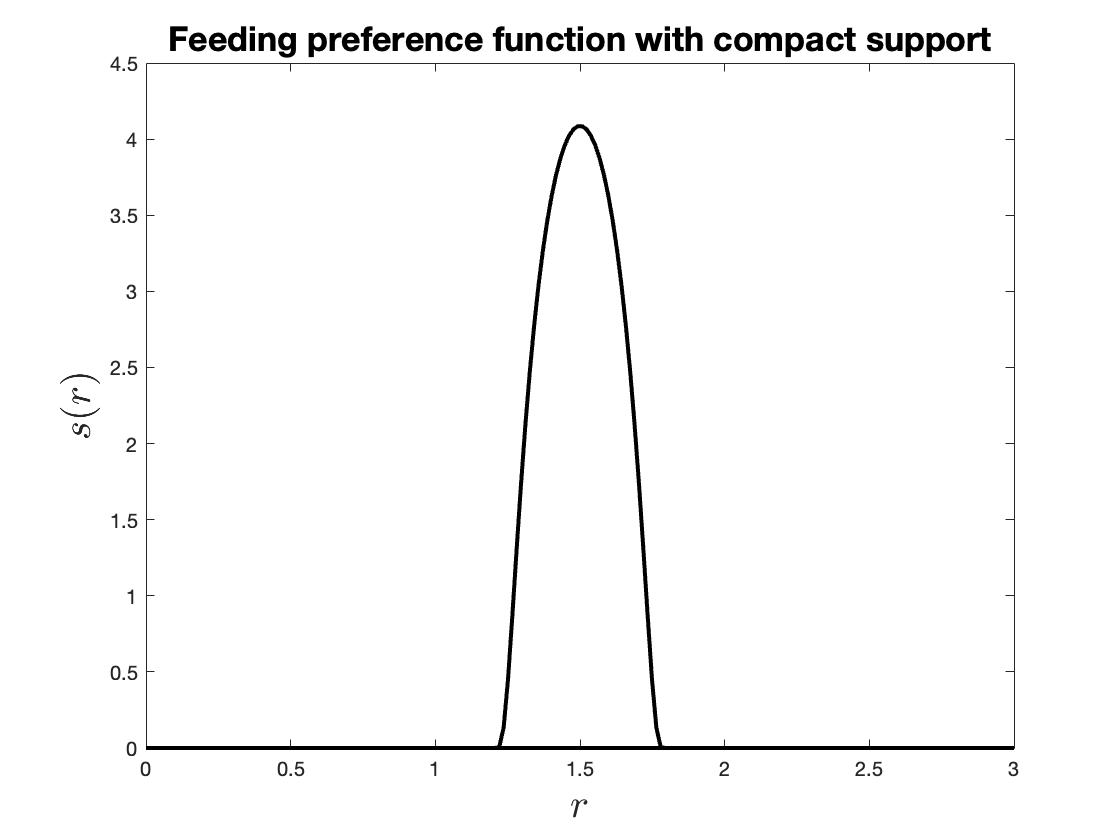}
	\end{subfigure}
	\caption{Both choices of the feeding preference function  $s(\cdot)$ plotted against predator-prey ratio $r$. On the left, the Gaußian feeding preference \eqref{d:sgauss}. On the right the feeding preference with compact support \eqref{d:scomp}. The parameters $B=1.5$ and $\sigma=0.3$ were chosen.}
	\label{F:s}
\end{figure}

We point out that both feeding preference functions a-priori admit predation events where predators are of smaller size than their prey. Since these predation events are under-represented in aquatic food-webs a reasonable choice of parameters $B$ and $\sigma$ would include $B-\sigma \geq1$. In that case the probability of a predator feeding on prey bigger than itself is very unlikely in case of \eqref{d:sgauss}, while it is completely impossible with the choice \eqref{d:scomp}.

In the limit-case $\sigma \to 0$ the feeding preference function degenerates in both cases to
$$
s(r)=\delta_B(r),
$$
implying that a predator exclusively feeds on prey with size given by a fraction $1/B$ of its own size, which is inspired by \cite{SP}. The corresponding model \eqref{e:PJG}-\eqref{IC_PJG} then simplifies to the following ordinary differential equation
\begin{align}\label{e:PJGdirac}
	\pa_tf(w)=&\frac{B^{\alpha}w^{\alpha+1}}{(K+B)^{\alpha+2}} f\left(\frac{w}{B+K}\right)f\left(\frac{Bw}{B+K}\right) + \frac{B^{\alpha}(1-K)w^{\alpha+1}}{K'^{\alpha+3}}f\left(\frac{w}{K'}\right)f\left(\frac{Bw}{K'}\right) \notag \\
	\phantom{=}-&\frac{w^{\alpha+1}}{B^2}f(w)f\left(\frac{w}{B}\right)-B^{\alpha}w^{\alpha+1}f(w)f(Bw), \\
	f(w,0)&=f_0, \notag 
\end{align} 
for $w \in [0,\infty)$. 

\paragraph{Assumptions on dependencies:}\label{p:parameters}
The parameters and quantities model \eqref{e:PJG}-\eqref{IC_PJG} depends on are broad features of an ecosystem and cannot be measured for a general community explicitly. Rather, we summarize our assumptions on the aforementioned, see Table \ref{t:parameters}, and discuss their meaning, considering biological relevance and comparing to existing literature on aquatic organisms and food webs.  

\begin{center}
	\begin{tabular}{||c | c | c ||} 
		\hline
		\textbf{Parameter} & \textbf{Parameter-range considered} & \textbf{Definition} \\ [0.5ex] 
		\hline
		$K$ & 0.1-0.6 & Assimilation constant  \\ 
		\hline
		$K'$ & 0.001-0.1 & 'Offspring' production constant  \\
		\hline
		$\alpha$ & 0.5-1.5 & Search volume exponent  \\
		\hline
		$A$ & 1 & Volume searched per mass$^{-\alpha}$  \\
		\hline
		$B$ & $\geq$ 1 & Preferred ratio of predator/prey size  \\  
		\hline
		$\sigma$ & $\leq$ 0.5 & Diet breadth (Variance from $B$) \\ [1ex] 
		\hline 
	\end{tabular}
	\label{t:parameters}
\end{center}
Throughout this article, we investigate the parameters' influence on the dynamics on the ecosystem by varying them within their parameter-ranges given in the table above, while also making use of the underlying mathematical structure given by specific values. The reference values are taken from previous works, close to biological measurements \cite{BR,DDL,PD,W}.


\subsection{Formal Properties}\label{ss:properties}

The weak formulation \eqref{e:weakBEp} in the setting of the deterministic jump-growth model with 'offspring' production can be written as 
\begin{align}\label{e:weakPJG}
	&\frac{\md}{\md t} \int_0^{\infty} f(w) \vp(w) \, \md w \\
	&= \int_0^{\infty} \int_0^{\infty} \left(\vp\left(w+Kw'\right)+ \frac{1-K}{K'} \vp\left(K'w'\right)  - \vp(w) - \vp(w')\right) w ^{\alpha} s\left(\frac{w}{w'}\right) f(w) f(w') \, \md w' \, \md w,\notag
\end{align}
where again $\vp : \R_+\to\R_+$ describes a suitable test-function and $s(\cdot)$ can either be of the form \eqref{d:sgauss} or \eqref{d:scomp} at that stage. Formal \emph{conservation of total biomass} of the system \eqref{d:biomass} becomes clear with the choice $\vp(w)=w$ in \eqref{e:weakPJG}, since the expression under the brackets within the integrand vanishes:
\begin{align}\label{c:biomasscons}
	\frac{\md}{\md t} \int_0^{\infty} w f(w) \, \md w = \int_0^{\infty} \int_0^{\infty} \left((w + w')  - (w+w')\right) w^{\alpha} s\left(\frac{w}{w'}\right) f(w) f(w') \, \md w' \, \md w \,.
\end{align}
It is important to state that at that stage, the first moment is just a formally conserved quantity of the system, induced by the modelled post-predatory weight distribution on an individual based scale. The right-hand-side is indeed equal to zero provided that
$$
	\int_0^{\infty} \int_0^{\infty} (w + w')  w^{\alpha} s\left(\frac{w}{w'}\right) f(w,t) f(w',t) \, \md w' \, \md w < \infty, \quad \forall t>0.
$$
However, a blow-up of the of total biomass can happen in finite time if the above condition fails to hold. This relates to the \emph{gelation phenomenon}, which is by now well understood in the framework of coagulation equations (see \cite{EMP, EP, GL1, GL2} and references therein). 

Further, the choice $\vp \equiv 1$ for the test-function gives insights over the dynamics of the \emph{total number of organisms} in the ecosystem
\begin{align}\label{d:number}
	\mathcal{N}(t):= \int_0^\infty f(w,t) \, \md w.
\end{align}
We obtain from \eqref{e:weakPJG} with $\vp \equiv 1$ 
\begin{align*}
	\dot{\mathcal{N}}(t) = \frac{1-K-K'}{K'} \int_0^{\infty} \int_0^{\infty} w ^{\alpha} s\left(\frac{w}{w'}\right) f(w) f(w') \, \md w' \, \md w,\notag
\end{align*}
hence, increase of the number of individuals in time due to the production of a huge amount of small-sized organisms as a by-product of a predation event. Moreover, although not directly relevant for our model assumptions, the case $s \equiv \bar{s}$ and $\alpha=0$ gives the following differential relation:
\begin{align*}
	\dot{\mathcal{N}}(t) = \bar{s}\frac{1-K-K'}{K'} \mathcal{N}^2(t)\,.
\end{align*}
Thus, $\mathcal{N}(t)$ blows-up in finite time $t_{max}=\frac{K'}{\bar{s}(1-K-K')\mathcal{N}_0}$.
\begin{rem}
The above observations already indicate that attention should be paid to a possible blow-up of moments in finite time. Indeed, from the formulations above it is not clear if the quantities $\m(t)$ and $\mathcal{N}(t)$ stay finite in finite time. However, if one can control the $\alpha$-th moment with respect to $t$, i.e.
\begin{align*}
	\m_{\alpha}(t):=\int_0^\infty w^{\alpha}f(w)\, \md w < \infty, \quad \forall t >0,
\end{align*}
it is easily seen that further $\mathcal{N}(t)$ is controlled due to the boundedness of the feeding preference function $s(r)\leq \bar{s}$, for all $r \in \R_+$. Indeed, in that case we can estimate
\begin{align*}
	\dot{\mathcal{N}}(t) \leq \frac{1-K-K'}{K'} \bar{s} \m_{\alpha}(t) \mathcal{N}(t), 
\end{align*}
and hence, by the virtue of Grönwall inequality, 
\begin{align*}
	\mathcal{N}(t) \leq \mathcal{N}(0) \exp{\left(\frac{1-K-K'}{K'} \bar{s} \int_0^t \m_{\alpha}(s)\, \md s\right)} < \infty, \quad \forall t>0.
\end{align*}
Moreover, the integrals on the right-hand-side of \eqref{c:biomasscons} are finite, hence the total biomass is conserved, provided that the $\m_{\alpha}(t)$ is finite and the feeding preference function $s(\cdot)$ has compact support, as in \eqref{d:scomp}. This can be seen by the following estimate
\begin{align*}
	\int_0^{\infty} \int_0^{\infty} (w + w')  w^{\alpha} s\left(\frac{w}{w'}\right) f(w) f(w') \, \md w' \, \md w \leq \bar{s}C(\sigma, B) \m_{\alpha}(t) \m,
\end{align*}
where we used that due to the boundedness of the support of $s(\cdot)$ one can always estimate the prey weight with the predator weight and vice versa. 
\end{rem}

\paragraph{Moment-Control:}\label{sss:moments}

We generalise the formal observations about the zeroth and the first moment, to the question which moments
\begin{align}\label{d:moment}
	\m_m[f](t):=\int_0^\infty w^m f(w,t) \, \md w
\end{align}
of the distribution $f(\cdot, t)$ are expected to be controlled over time $t>0$. Starting again from the weak formulation \eqref{e:weakPJG} and the choice $\vp(w):=w^{m}$, we first aim to identify the sign of the right-hand-side with respect to the power $m>0$. We obtain
\begin{align}\label{e:moments}
	&\dot{\m}_m[f](t)= \int_0^{\infty} \int_0^{\infty} \left(\left(w+Kw'\right)^m+ \frac{1-K}{K'} \left(K'w'\right)^m  - w^m - w'^m\right)w^{\alpha} s\left(\frac{w}{w'}\right) f(w) f(w') \, \md w' \, \md w \notag \\
	&=\int_0^{\infty} \int_0^{\infty} \left(\left(\frac{w}{w'}+K\right)^m+ (1-K)K'^{m-1}  - \left(\frac{w}{w'}\right)^m -1 \right)w^{\alpha}w'^m s\left(\frac{w}{w'}\right) f(w) f(w') \, \md w' \, \md w  \\
	&= \int_0^\infty \int_0^{\infty} \left(\left(r+K\right)^m+ (1-K)K'^{m-1}  - r^m - 1\right) w^{\alpha+m+1} r^{-m-2} s(r) f\left(\frac{w}{r}\right) f(w) \, \md w \, \md r, \notag
\end{align}
where similar to \eqref{e:PJG_r} the last equality was due to the coordinate change $w' \to r$, via $\frac{w}{w'}=r$. Negativity of the expression between the brackets indicates decay of the corresponding moment, provided the involved integrals are finite. Hence, we aim to characterize under which choice of parameters the expression between the brackets is negative for a.e. $(r, m) \in (0,\infty) \times (0,\infty)$. Therefore, we define the quantity
\begin{align}\label{d:F}
	F(m,r):= (r+K)^m+(1-K)K'^{m-1}-r^m-1,
\end{align} 
and investigate the behaviour of the function for parameters $K, K' \ll 1$ and $(m,r) \in \R_+\times\R_+$. 
First, one notices that for $m=0$, which corresponds to \eqref{d:number}, $F$ is independent of $r$ and has value 
$$
	F(0,r)=\frac{1-K-K'}{K'} >0, \quad \forall r \in [0,\infty),
$$
which increases as the 'offspring' production constant $K'$ decreases. Obviously, $F$ vanishes for $m=1$ for all $r \in [0,\infty)$, which coincides with formal conservation of total biomass \eqref{c:biomasscons}. Second, one easily calculates
$$
\pa_mF(m,r)=\ln{\left(r+K\right)}\left(r+K\right)^m +(1-K)\ln{(K')}K'^{m-1}-\ln(r)r^m,
$$
thus, the derivative of $F(\cdot,r)$ at $m=1$ is given by
$$
\pa_mF(1,r) = \ln{\left(r+K\right)}\left(r+K\right) +(1-K)\ln{(K')} - \ln(r)r.
$$
Its sign decides if we have decay of moments and whether they are greater or less than 1. By a simple reformulation of the above,  one obtains the condition
\begin{align}\label{c:der}
\pa_mF(1,r)\lessgtr0 \quad \Leftrightarrow \quad (r+K)^{r+K}K'^{1-K}\lessgtr r^r.
\end{align}
These observations lead to the following result.

\begin{lemma}\label{l:mm}
	Let $r \in [a,b]$ such that $0<a<b$. Then for $K<1$ given and $K'$ chosen accordingly small enough, there exist constants $0<\tilde{m}<1$ and $1<m_*<\infty$ such that we have
	\begin{align*}
		F(m,r) > 0, \quad \forall m \in (\tilde{m},1), \quad \forall  r \in [a,b],
	\end{align*}
	and
	\begin{align*}
		F(m,r) < 0, \quad \forall m \in (1, m_*), \quad \forall  r \in [a,b].
	\end{align*}	
\end{lemma}

\begin{proof}
Starting from condition \eqref{c:der}, we can observe that one can find sufficiently small $K'$ such that $(r+K)^{r+K}K'^{1-K} < r^r$ being equivalent to
$$
	\left(1+\frac{K}{r}\right)^r\left(r+K\right)^K K'^{1-K} < 1
$$  
holds for all $0<B-\sigma < r < B+\sigma$. This ensures the existence of an interval $[1,m_*]$, $m_*>1$, such that $F(m,r) \leq 0$ for all $m \in [1,m_*]$ and $r \in [B-\sigma,B+\sigma]$, from which claim \eqref{c:entropy} follows.

Moreover, from the strict inequality in condition \eqref{c:der} one can deduce that, under the same parameter regime for sufficiently small $K'$, there has to exist an interval $[\tilde{m},1]$, $\tilde{m}<1$, of small powers such that $F(m,r)>0$ for all $m \in [\tilde{m},1]$, $r \in [B-\sigma,B+\sigma]$ and, hence, growth in time of the moments smaller than 1 \eqref{c:growth}.
\end{proof}

\begin{figure}[H]
	\centering
	\includegraphics[width=10cm]{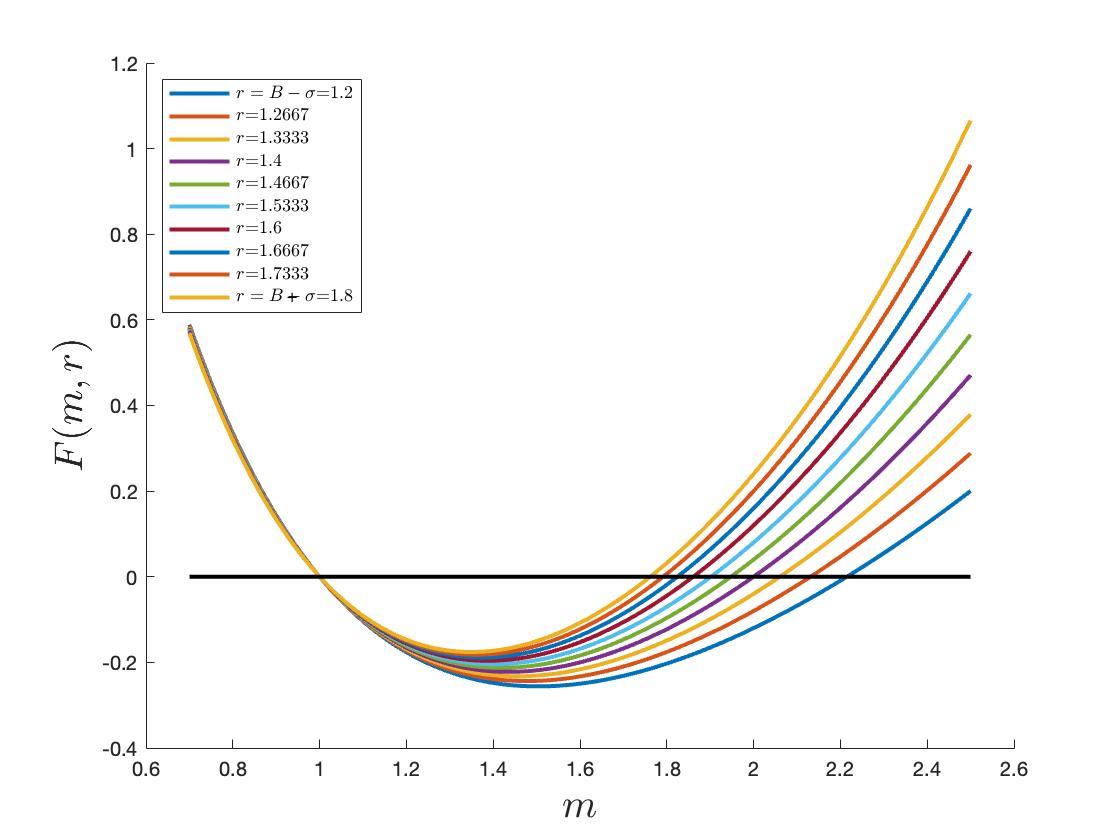}
	\caption{Function $F$ \eqref{d:F} plotted against moments $m$ with various values $r \in \supp{(s)}$, with  $s(\cdot)$ as \eqref{d:scomp}, $B=1.5$, $\sigma=0.3$, $K=0.3$ and $K'=0.1$.}
	\label{F(m,r)}
\end{figure}

\begin{rem}\label{r:F}
	\begin{itemize}
	 \item The above results indicate growth and decay of the moments of a solution $f:\, \R_+\times\R_+ \to \R_+$ to \eqref{e:PJG}-\eqref{IC_PJG}. Indeed, if the feeding preference function $s:\,\R_+\to\R_+$ is given as \eqref{d:scomp} such that $\inf \supp{(s)}=B-\sigma >0$. Then for $K$ and $K'$ small enough, we have
	\begin{align}\label{c:growth}
		\dot{\m}_m[f](t) \geq 0, \quad \forall m \in (\tilde{m},1),
	\end{align}
	and
	\begin{align}\label{c:entropy}
		\dot{\m}_m[f](t) \leq 0, \quad \forall m \in (1, m_*),
	\end{align}
	provided, that we do not deal with infinite integrals on the right-hand-side of \eqref{e:moments}. This is the case for $\alpha \in (1,m_*)$ with $\m_{\alpha}(0)<\infty$, from which we can deduce that the moment $\m_{\alpha}(t)$ is non-increasing. This further implies the uniform upper-bound
	$$
		\m_m[f](t) \leq \m_m(0) := \int_0^\infty w^m f_0(w)\, \md w, \quad \forall t >0, \quad m \in (1,m_*).
	$$
	\item The constants $m_*$ and $\tilde{m}$, borders of the intervals, where we can associate a definitive sign to $F$, depend only on the interval of the feeding ratio $r$ as well as on the small parameters $K$ and $K'$. Hence, $K, K'$ and the shape of $s(\cdot)$ \eqref{d:scomp} via its parameters $\sigma$ and $B$ define the moments which are growing and which are decaying. As is customary for this class of problems a tractable, explicit formula for the solution is a-priori completely out of reach. However, our analysis of the qualitative features of the model does not require such an explicit solution.
	
	 In Figure \ref{F(m,r)} the function $F$ \eqref{d:F} was plotted for $m \in [0.7,2.5]$ and various values for $r \in [B-\sigma,B+\sigma]$, where $B=1.5$ and $\sigma=0.3$ was chosen. One can see that $m_* \approx 1.75$ is defined by $F(m_*,B-\sigma)=0$, while $\tilde{m}$ is clearly negative, hence $F(m,r)>0$ for all $m \in [0,1)$, $r \in [B-\sigma,B+\sigma]$.
	 \end{itemize}
\end{rem}

\section{Existence of Solutions}\label{s:existence}

The existence results presented rely on fixed-point arguments of the integral operator on the right-hand-side of \eqref{e:PJG}. Due to the formal conservation of the first moment \eqref{c:biomasscons} we expect the space 
\begin{align*}
	L^1(\R_+,w):=\left\{f : \, \R_+ \to \R_+:\, wf(w) \in L^1(\R_+) \right\}
\end{align*}
to be the right setting. Model-specific properties of the feeding kernel \eqref{g:kernel} are highly relevant: Compactness of the support of the feeding preference $s(\cdot)$ is needed. Moreover, the previously established control of the $\alpha$-th moment in Section \ref{sss:moments}, which just holds for $\alpha \in (1,m_*)$, indicates how crucial its parameter range is to the type of existence result we obtain. Indeed, it is already well known (see, e.g.  \cite{EMP, EP, GL1, GL2}) that due to a possible blow-up of the solution in specific parameter regimes of the exponent one can expect global or just local in time existence of a solution.

\begin{theorem}\label{T:existence1}
Let the feeding preference function $s:\,\R_+\to\R_+$ be given as \eqref{d:scomp} such that $\inf \supp{(s)}=B-\sigma >0$ and $K$ and $K'$ small enough. We further assume that for the search-volume exponent it holds $\alpha \in (1,m_*)$ and $f_0 \in L^1_+(\R_+,w)$ such that 
$$
	\m_{\alpha}^0:=\int_0^{\infty}w^{\alpha} f_0(w) \, \md w < \infty.
$$
Then there exists a unique solution $f \in C^1\left([0,\infty),L^1_+(\R_+,w) \right)$ to \eqref{e:PJG}-\eqref{IC_PJG} and \newline $\m_{\alpha}(t):=\int_0^{\infty}w^{\alpha} f(w,t) \, \md w \leq \m_{\alpha}^0$.
\end{theorem}	

\begin{proof}
We aim to prove a \emph{Lipschitz}-property of the right-hand-side of \eqref{e:PJG}. Let therefore $f^1, f^2 \in L^1_+(\R_+,w)$ with 
$$
\|w^{\alpha-1}f^1\|_{L^1_+(\R_+,w)},\|w^{\alpha-1}f^2\|_{L^1_+(\R_+,w)} \leq \m_\alpha^0.
$$
The four terms on the right-hand-side of \eqref{e:PJG} will be dealt with separately, where always the trivial algebraic identity $f^1 f^{1'} - f^2f^{2'} = f^1(f^{1'} - f^{2'})+f^{2'}(f^1-f^2)$ is used. For the first gain term $G_1$ we obtain
\begin{align*}
	&\left\|G_1(f^1,f^1)-G_1(f^2,f^2)\right\|_{L^1(\R_+,w)} \\
	&=\left\| \int_0^{\frac{w}{K}} \left(w-Kw'\right)^{\alpha} s\left(\frac{w-Kw'}{w'}\right) \left(f^1(w-Kw') f^1(w') - f^2(w-Kw') f^2(w')\right) \, \md w' \right\|_{L^1\left(\R_+,w\right)} \\
	& \leq \int_0^{\infty} \int_0^{\infty} w \left(w-Kw'\right)_+^{\alpha} s\left(\frac{w-Kw'}{w'}\right) \left|f^1(w-Kw')\right| \left|f^1(w')-f^2(w')\right| \,\md w' \, \md w \\
	&+  \int_0^{\infty} \int_0^{\infty}w \left(w-Kw'\right)_+^{\alpha} s\left(\frac{w-Kw'}{w'}\right) \left|f^2(w')\right| \left|f^1(w-Kw')-f^2(w-Kw')\right|  \,\md w' \, \md w \\
	&= \frac{1}{K} \int_0^{\infty} \int_0^{\infty} w \left(w-w'\right)_+^{\alpha} s\left(K\frac{w-w'}{w'}\right) |f^1(w-w')| \left|f^1\left(\frac{w'}{K}\right)-f^2\left(\frac{w'}{K}\right)\right| \,\md w' \, \md w \\
	&+ \frac{1}{K} \int_0^{\infty} \int_0^{\infty}w \left(w-w'\right)_+^{\alpha}  s\left(K\frac{w-w'}{w'}\right) \left|f^2\left(\frac{w'}{K}\right)\right| |f^1(w-w')-f^2(w-w')|  \,\md w' \, \md w.
\end{align*}
Next, we use the boundedness of the feeding preference function $s(\cdot)$ \eqref{d:scomp} as well as its compact support, so that we can always estimate the predator weight with the prey and vice versa. The to these estimates corresponding multiplicative constants will not be written explicitly and we notice that they have different forms depending on the term of the collision operator which is estimated. Starting with the first gain term we obtain
\begin{align*}
	&\left\|G_1(f^1,f^1)-G_1(f^2,f^2)\right\|_{L^1(\R_+,w)} \\
	&\leq \frac{C_1(B,\sigma,K)}{K} \int_0^{\infty} \int_0^{\infty} w' \left(w-w'\right)_+^{\alpha} \left|f^1(w-w')\right| \left|f^1\left(\frac{w'}{K}\right)-f^2\left(\frac{w'}{K}\right)\right| \,\md w' \, \md w \\
	&+\frac{C_2(B,\sigma,K)}{K} \int_0^{\infty} \int_0^{\infty}w'^{\alpha} \left(w-w'\right)_+ \left|f^2(w')\right| \left|f^1(w-w')-f^2(w-w')\right|  \,\md w' \, \md w \\
	&\leq K C_1(B,\sigma,K) \m_\alpha^0  \|f^1-f^2\|_{L^1(\R_+,w)} + \frac{C_2(B,\sigma,K)}{K} \m_\alpha^0  \|f^1-f^2\|_{L^1(\R_+,w)} \\
	&\leq \tilde{C}(B,\sigma, K) \m_\alpha^0  \|f^1-f^2\|_{L^1(\R_+,w)}\,,	
\end{align*}
where $\tilde{C}(B,\sigma,K) = \max\left\{ C_1(B,\sigma,K),  C_2(B,\sigma,K)\right\}$. Similarly we proceed with the remaining three terms:
\begin{align*}
	&\left\|G_2(f^1,f^1)-G_2(f^2,f^2)\right\|_{L^1(\R_+,w)} \\
	&=\frac{1-K}{K'^2}\left\| \int_0^{\infty}w'^{\alpha} s\left(\frac{w'K'}{w}\right) \left(f^1\left(\frac{w}{K'}\right) f^1(w') - f^2\left(\frac{w}{K'}\right) f^2(w')\right) \, \md w' \right\|_{L^1\left(\R_+,w\right)} \\
	&\leq C_3(B,\sigma,K') (1-K)K'^{\alpha-1} \int_0^{\infty} \int_0^{\infty} w^{\alpha}w' \left|f^1\left(w\right)\right| \left|f^1(w')-f^2(w')\right| \, \md w' \, \md w \\
	&+ C_4(B,\sigma) (1-K)K'^{\alpha-1} \int_0^{\infty} \int_0^{\infty} w w'^{\alpha} \left|f^2\left(w'\right)\right| \left|f^1\left(w\right) -f^2\left(w\right) \right| \, \md w' \, \md w \\
	&= \hat{C}(B,\sigma,K') \m_\alpha^0 \|f^1-f^2\|_{L^1(\R_+,w)},
\end{align*}
where again $\hat{C}(B,\sigma,K')$ describes the maximum over the two multiplicative constants before. The loss terms are due to the lack of convolution of a simpler structure:
\begin{align*}
	&\left\|L_1(f^1,f^1)-L_1(f^2,f^2)\right\|_{L^1(\R_+,w)} \\
	&=\left\| \int_0^{\infty}w^{\alpha} s\left(\frac{w}{w'}\right) \left(f^1(w) f^1(w') - f^2(w) f^2(w')\right) \, \md w' \right\|_{L^1\left(\R_+,w\right)} \\
	&\leq C_4(B,\sigma)  \int_0^{\infty} \int_0^{\infty} w^{\alpha}w' \left|f^1(w)\right| \left|f^1(w')-f^2(w')\right| \, \md w' \, \md w \\
	&+ C_4(B,\sigma) \int_0^{\infty} \int_0^{\infty} w w'^{\alpha} \left|f^2\left(w'\right)\right| \left|f^1(w) -f^2(w) \right| \, \md w' \, \md w \\
	&=  2 C_4(B,\sigma)  \m_\alpha^0 \|f^1-f^2\|_{L^1(\R_+,w)},
\end{align*}
and
\begin{align*}
	&\left\|L_2(f^1,f^1)-L_2(f^2,f^2)\right\|_{L^1(\R_+,w)} \\
	&=\left\| \int_0^{\infty}w'^{\alpha} s\left(\frac{w'}{w}\right) \left(f^1(w) f^1(w') - f^2(w) f^2(w')\right) \, \md w' \right\|_{L^1\left(\R_+,w\right)} \\
	&\leq  2 C_4(B,\sigma)  \m_\alpha^0 \|f^1-f^2\|_{L^1(\R_+,w)}.
\end{align*}
Unifying the above estimates, we obtain
\begin{align*}
	\|Q(f^1, f^1)-Q(f^2, f^2)\|_{L^1_+(\R_+,w)} \leq C(B, \sigma, K, K') \m_\alpha^0 \|f^1-f^2\|_{L^1(\R_+,w)}\,.
\end{align*}
Therefore a unique local solution exists by Picard iteration, which preserves non-negativity. Boundedness of the $\alpha$-moment holds due to $\alpha \in (1,m_*]$ and \eqref{c:entropy}. In this regime, the formal conservation of biomass, i.e. of the $L^1(\R_+,w)$-norm, holds rigorously, implying global existence. This is a consequence of the boundedness of $\m_{\alpha}$, which implies together with the boundedness and compact support of $s(\cdot)$ that all integrals in \eqref{c:biomasscons} are finite. Hence, no finite-time blow-up phenomenon is possible. 
\end{proof}

The \emph{search exponent} $\alpha$ highly depends on the \emph{trophic level} the organisms are in, but is always given by a number close to 1 and in many modelling approaches taken equal to 1, see, e.g., \cite{DDL}. Suggested by Ware in '78 \cite{W}, for \emph{pelagic fish} $\alpha  \in [0.6,0.9]$ is biologically reasonable. In that case, however, it can be assumed that $\tilde{m}<\alpha<1$, hence \eqref{c:growth} implies growth of $\m_{\alpha}$ with respect to time. Moreover, a \emph{blow-up in finite time} of $\m_{\alpha}(t)$ is viable. From \eqref{e:moments} we estimate due to the boundedness of  $s(\cdot)$ and $F$ 
\begin{align*}
	\dot{\m}_{\alpha}(t) &= \int_0^\infty \int_0^\infty F\left(\alpha, \frac{w}{w'}\right)w^\alpha w'^\alpha s\left(\frac{w}{w'}\right) f(w)f(w') \, \md w' \, \md w \leq C(B,\sigma, K, K') \m_{\alpha}^2(t),
\end{align*}
where $C(B,\sigma, K, K')$ describes a multiplicative constant depending on the parameters in the arguments. This leads to 
\begin{align}\label{c:blowup}
	\m_{\alpha}(t) \leq \frac{\m_{\alpha}(0)}{1-t C(B,\sigma, K, K')\m_{\alpha}(0)}.
\end{align}
Hence, boundedness of $\m_{\alpha}(t)$ can be concluded up to any time $t<\frac{1}{C(B,\sigma, K, K')\m_{\alpha}(0)}$, from which the following local-in-time existence result can be deduced.

\begin{theorem}\label{T:existence2}
	Let the assumptions of Theorem \ref{T:existence1} hold, but with search volume exponent $\alpha \in (\tilde{m},1)$ and assume $f_0 \in L_+^1(\R_+,w)$ such that
	$$
		\m_\alpha^0:=\int_0^\infty w^{\alpha}f_0(w) \, \md w < \infty.
	$$
	Then there exists a unique solution $f \in C^1\left(\left[0, \bar{T}\right),L_+^1(\R_+,w)\right)$ to \eqref{e:PJG}-\eqref{IC_PJG}, where the upper bound for the time interval is given by 
	$$
		0 < \bar{T} :=\frac{1}{\m_\alpha^0C(B,\sigma, K, K')}.
	$$
\end{theorem}
\begin{proof}
	We start by explicitly stating the fixed-point procedure
	\begin{align}\label{FP}
		\mathcal{FP}_t: L^1_+(\R_+,w) \to L^1_+(\R_+,w)\,,\quad f \mapsto f_0 + \int_0^t Q(f,f) \, \md s\,,
	\end{align} 
	with $f_0 \in L^1_+(\R_+,w)$ chosen such that
	$$
	\|w^{\alpha-1}f_0\|_{L^1_+(\R_+,w)} = \m_\alpha^0 < \bar{\m}_{\alpha}^T < \infty\,.
	$$
	Due to \eqref{c:blowup} we expect that the $\alpha$-th moment $\m_{\alpha}(t)$ of a solution to \eqref{e:PJG}-\eqref{IC_PJG} blows up in finite time. Therefore, we fix an arbitrary $T>0$ such that $T <\bar{T}$, ensuring that the fraction 
	$$
		 \frac{\m_{\alpha}^0}{1-T C(B,\sigma, K, K')\m_{\alpha}^0}=:\bar{\m}_{\alpha}^T
	$$
	is finite. It can easily be verified that for $f  \in L^1_+(\R_+,w)$ fulfilling $\|w^{\alpha-1}f_0\|_{L^1_+(\R_+,w)} = \m_{\alpha}^0 < \bar{\m}_{\alpha}^T$ we have
	$$
		\int_0^\infty w^{\alpha}\mathcal{FP}_t (f) \, \md w \leq \m_{\alpha}^0 + t C(B,\sigma, K, K') (\bar{\m}_{\alpha}^T)^2\,.
	$$
	Hence, this property is propagated in the fixed-point procedure \eqref{FP}, when the time $t<T$ is chosen small enough such that
	$$
		t < \frac{\bar{\m}_{\alpha}^T-\m_{\alpha}^0}{C(B,\sigma, K, K') (\bar{\m}_{\alpha}^T)^2}\,.
	$$
	The finiteness of the $\alpha$ moment within the fixed-point procedure ensures further that the first moment remains preserved, while positivity follows from the structure of $Q(\cdot,\cdot)$. Following the estimates of $\|Q(f^1,f^1)-Q(f^2,f^2)\|_{L^1(\R_+,w)}$ in the proof of Theorem \ref{T:existence1}, we see (by choosing $t$ probably even smaller) that $\mathcal{FP}_t$ is a contraction on $L^1_+(\R_+,w)$. In virtue of Banach's fixed point theorem a unique solution $f \in C^1\left(\left[0, t\right),L_+^1(\R_+,w)\right)$ exists. By Picard-iteration the time-interval can be extended as long as $t < T<\bar{T}$ ensuring $\m_{\alpha}(t) <\bar{\m}_{\alpha}^T$.
\end{proof}

\section{Steady States and Long-time Behaviour}\label{s:longtime}

In aquatic ecosystems, abundance of organisms with respect to their body-size averaged over space and seasonal changes often varies rather little, suggesting that they may be close to a steady state. We aim to identify the equilibria and long-time behaviour of solutions to our model \eqref{e:PJG}-\eqref{IC_PJG}, before discussing their biological interpretation and significance. 

\subsection{Stationary States}\label{s:ss}

\paragraph{Trivial steady state:} By inserting one can see easily, that at the formal state
\begin{align}\label{d:zerostate}
	\bar{f}_0(w) = \frac{\m}{w}\delta_0(w)
\end{align}
the right-hand-side of \eqref{e:PJG} evaluates at zero. Hence, the model admits a \emph{trivial stationary state}, compatible with conservation of total biomass \eqref{c:biomasscons} via the factor $\frac{\m}{w}$ before $\delta_0$, the Dirac distribution centred at 0. This equilibrium, however, represents a completely extinct ecosystem consisting just of an infinite amount of microorganisms. 

\paragraph{Non-trivial steady state - gaps in the size-spectrum:} In virtue of formulation \eqref{e:PJG_r}, a sufficient condition for $f$ being at a non-trivial steady state is given by
\begin{align}\label{e:stat}
	0=&r^{\alpha}(r+K)^{-\alpha-2} f\left(\frac{wr}{r+K}\right)f\left(\frac{w}{r+K}\right) + (1-K)K'^{-3-\alpha} r^{\alpha}f\left(\frac{w}{K'}\right)f\left(\frac{wr}{K'}\right) \notag \\
	-&r^{-2} f(w) f\left(\frac{w}{r}\right) - r^{\alpha} f(w)f(rw) \notag \\
	& \\
	\m =& \int_0^\infty w f(w) \, \md w, \notag
\end{align}
for almost all $r \in \supp{(s)}$ and almost all $w \in \R_+$. From \eqref{e:stat} it is easy to deduce a \emph{sufficient condition} a steady state solution has to fulfil, given by
$$
	f(w)f(rw) = 0, \quad \text{for a.a. } \,\, r \in \supp{(s)}, w \in \R_+.
$$
On the other hand, in case of $\supp{(s)}=\R_+$, as for \eqref{d:sgauss}, this can only be fulfilled for a distribution such that $f(w)=0$ for almost all $w \neq 0$, which again leads us to the trivial steady state discussed above \eqref{d:zerostate}. On the other hand, if $\supp{(s)}=[B-\sigma,B+\sigma]$, as \eqref{d:scomp}, an admissible steady state, given by a \emph{size spectrum function with gaps}, is possible. To be more precise, the distribution function $f$ has to satisfy the following condition for a $w \in \R_+$:
\begin{align}\label{c:gaps}
	f(w) \neq 0 \, \Rightarrow f(w')=0, \, \text{for a.e. }\, w' \in \left[\frac{w}{B+\sigma}, \frac{w}{B-\sigma}\right].
\end{align}
Obviously, parameters $B$ and $\sigma$ fulfilling
$$
	B-\sigma<1<B+\sigma
$$
induce $w \in  \left[\frac{w}{B+\sigma}, \frac{w}{B-\sigma}\right]$, thus, \eqref{e:stat} is again only satisfied by the trivial equilibrium \eqref{d:zerostate}. The condition
\begin{align}\label{c:Bs}
	B-\sigma>1 \quad \text{or} \quad B+\sigma<1
\end{align}
ensures that $w \notin \left[\frac{w}{B+\sigma}, \frac{w}{B-\sigma}\right]$, hence the existence of a non-trivial equilibrium is possible. In an ecological context this means that an organism does not predate on other organisms within its own trophic level or the next lower or higher ones, closest to it. Moreover, special forms of such steady states can be computed by identifying the gaps' sizes with respect to the preferred feeding ratio $B$ and the variance $\sigma$. Taking into account the biological reasoning, we restrict ourselves to the case $B-\sigma>1$, implying that organisms exclusively feed on prey of smaller size. The support of a non-trivial steady sate $\bar{f}$ of \eqref{e:PJG}-\eqref{IC_PJG} with  $s(\cdot)$ given by \eqref{d:scomp} can be enclosed in the following infinite union of intervals
$$
	\supp{(\bar{f})} \subset \bigcup_{i \in \mathbb{Z}} \left[\frac{\bar{w}}{(B-\sigma)^{i+1}(B+\sigma)^i}, \frac{\bar{w}}{(B-\sigma)^i(B+\sigma)^i}\right],
$$ 
for a value $\bar{w} > 0$. The choice of the reference body size $\bar{w}$ is not unique. Indeed, feasible values are influenced by the initial conditions and model-specific parameters. A trivial observation indicates that the lengths of the intervals 
$$
	l_i:=\bar{w}\frac{B-\sigma-1}{(B-\sigma)^{i+1}(B+\sigma)^i}
$$ 
increase as $i$ decreases. Thus, the closer the intervals lie to the degenerate state zero, the smaller they have to be, which further leads to the conclusion that we expect a highly oscillatory behaviour of a steady state solution to \eqref{e:PJG}-\eqref{IC_PJG}. A further important observation is that the size of the intervals increases as $B$ increases, implying that the smaller the average prey size compared to the predator size becomes the bigger the holes in the size spectrum we can expect.

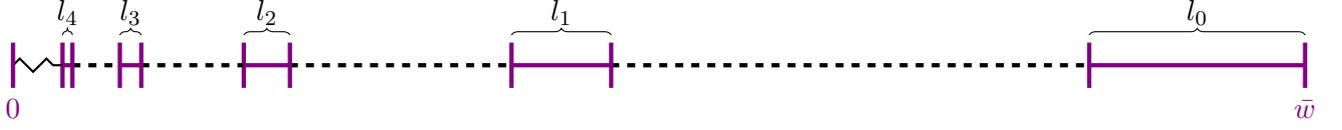
\begin{figure}
	\centering
	\begin{tikzpicture}
		\draw[ultra thick, -, violet] (17,0.3) -- (17,-0.3) node[anchor=south, below]{$\bar{w}$};
		\draw [ultra thick, - , violet] (14.16,0) -- (17,0);
		\draw [- , decorate, decoration=brace] (14.16,0.4) -- (17,0.4);
		\node[anchor=north, above] at (15.58,0.4){$l_0$};
		\draw[ultra thick, -, violet] (14.16,0.3) -- (14.16,-0.3);
		\draw [ultra thick, - , dashed] (7.87,0) -- (14.16,0);
		\draw[ultra thick, -, violet] (7.87,0.3) -- (7.87,-0.3);
		\draw [ultra thick, - , violet] (6.559,0) -- (7.87,0);
		\draw [- , decorate, decoration=brace] (6.559,0.4) -- (7.87,0.4);
		\node[anchor=north, above] at (7.2145,0.4){$l_1$};
		\draw[ultra thick, -, violet] (6.559,0.3) -- (6.559,-0.3);
		\draw [ultra thick, - , dashed] (3.643,0) -- (6.559,0);
		\draw[ultra thick, -, violet] (3.643,0.3) -- (3.643,-0.3);
		\draw [ultra thick, - , violet] (3.036,0) -- (3.643,0);
		\draw [- , decorate, decoration=brace] (3.036,0.4) -- (3.643,0.4);
		\node[anchor=north, above] at (3.3395,0.4){$l_2$};
		\draw [ultra thick, - , violet] (3.036,0.3) -- (3.036,-0.3);
		\draw [ultra thick, - , dashed] (1.687,0) -- (3.036,0);
		\draw [ultra thick, - , violet] (1.687,0.3) --(1.687,-0.3);
		\draw [ultra thick, - , violet] (1.4056,0) -- (1.687,0);
		\draw [- , decorate, decoration=brace](1.4056,0.4) -- (1.687,0.4);
		\node[anchor=north, above] at (1.5463,0.4){$l_3$};
		\draw [ultra thick, - , violet] (1.4056,0.3) --(1.4056,-0.3);
		\draw [ultra thick, - , dashed] (0.781,0) -- (1.4056,0);
		\draw [ultra thick, - , violet] (0.781,0.3) --(0.781,-0.3);
		\draw [ultra thick, - , violet] (0.651,0) -- (0.781,0);
		\draw [- , decorate, decoration=brace](0.651,0.4) -- (0.781,0.4);
		\node[anchor=north, above] at (0.716,0.4){$l_4$};
		\draw [ultra thick, - , violet] (0.651,0.3) --(0.651,-0.3);
		\draw[thick, -, decorate, decoration=zigzag] (0,0) -- (0.651,0);
		\draw[ultra thick, -, violet] (0,0.3) -- (0,-0.3) node[anchor=south, below]{0};
	\end{tikzpicture}
	\caption{Support of non-trivial steady state represented with solid lines, while the gaps in the size spectrum are given by dashed black lines. For this demonstration the values $B=1.5$, $\sigma=0.3$ and the reference weight $\bar{w}=17$ were chosen.}
	\label{f:cascades}
\end{figure}

\begin{rem}
	Due to the quantity of the parameters included in the model as well as due to its non-linear structure it turns out that in practice it is a highly complicated task to predict the shape of the feasible set of such reference values $\bar{w}$. We want to point this out by investigating a specific, but biologically very reasonable case. 
	
	Therefore, let the initial condition \eqref{IC_PJG} be such that $\|f_0\|_{L^1(\R,w)} = \m<\infty$, \newline $w_0:=\max\{\supp{(f_0)}\} < \infty$ such that $w_0f_0(w_0)>c>0$, i.e. we assume that initially there exists a group of largest species with size $w_0$ in the ecosystem of consideration, which inhibit a non-negligible amount of biomass. The feeding behaviour \eqref{d:scomp}, \eqref{c:Bs} suggests that the system allows a non-trivial steady state $\bar{f}$ with $w_0<\bar{w}:=\max\{\supp{(\bar{f})}\} < \infty$, which serves as reference value to determine the support of $\bar{f}$ as
	\begin{align}\label{c:gaps}
		\supp{(\bar{f})} \subset \bigcup_{i \geq 0} \left[\frac{\bar{w}}{(B-\sigma)^{i+1}(B+\sigma)^i}, \frac{\bar{w}}{(B-\sigma)^i(B+\sigma)^i}\right].
	\end{align}
	The question of existence and explicit form of $\bar{w}$ is non-trivial. Indeed, by defining $w_m(t) := \max\{\supp{(f(t))}\}$ and differentiating the total biomass \eqref{c:biomasscons}
	\begin{align*}
		\m = \int_0^{w_m(t)} w f(w) \, \md w
	\end{align*}
	with respect to time one obtains the following time-evolution of the largest organism-size of the ecosystem at time $t>0$:
	\begin{align*}
		\dot{w}_m(t) =& - \frac{1}{w_m(t)f(w_m(t),t)} \int_0^{w_m(t)} w Q(f,f)(w,t) \, \md w. \\
		w_m(0)=&w_0,
	\end{align*}
	Hence, the quantity 
	\begin{align*}
		\int_0^{w_m(t)} w Q(f,f)(w,t) \, \md w =& \int_0^{\infty}\int_0^{\infty} \left[(w+Kw')\mathlarger{\mathlarger{\mathbb{1}}}_{w+Kw' \leq w_m(t)} + (1-K)w' \mathlarger{\mathlarger{\mathbb{1}}}_{K'w' \leq m(t)} \right.\\
		&\left.- w \mathlarger{\mathlarger{\mathbb{1}}}_{w\leq m_m(t)}-w'\mathlarger{\mathlarger{\mathbb{1}}}_{w'\leq w_m(t)}\right] w^{\alpha}s\left(\frac{w}{w'} \right)f(w,t)f(w',t) \, \md w' \, \md w
	\end{align*} 
	can be interpreted as the flux of the system at the moving boundary $w_m(t)$, defining its evolution in time. Due to the complexity of the expression, a steady state $\bar{w}$, corresponding to the maximal organism-size of the stationary distribution $\bar{f}$, can be given only implicitly by 
	\begin{align*}
		\int_0^{\bar{w}} w Q(f,f)(w,t) \, \md w =0.
	\end{align*}	
\end{rem}

\begin{rem}\label{r:cascade}
This suppression of specific trophic levels in an ecosystem is a phenomenon known in ecology as \emph{trophic cascades}, which after being popularized in the '60s \cite{HSS}, was observed in a wide range of ecosystems around the world \cite{Eetal, RGK}. It describes the effect of indirect influence of one trophic level to the next lower/higher after the one of their primary prey/predators, known as \emph{top-down/bottom up trophic cascade}. An example for occurrence of a \emph{top-down cascade} is given if predators in trophic level $T_1$ show that much efficiency in predation that the abundance of their prey in the next lower trophic level $T_2$ decreases immensely. Extinction of individuals in $T_2$ can be expected, while growth of abundance of their prey in trophic level $T_3$ can be observed due to the release from predation pressure. These phenomena will iterate subsequently in all trophic levels giving the characteristic \emph{domes} followed by \emph{gaps} in the size spectrum.
\end{rem}

\subsection{Long-time Behaviour}\label{ss:longtime}
From \eqref{e:moments} and Lemma \ref{c:entropy} we remember that for a solution $f$ to \eqref{e:PJG}-\eqref{IC_PJG} with  $s(\cdot)$ defined as \eqref{d:scomp} and a power $m \in (1,m_*)$ we have
\begin{align}\label{i:decay}
	\dot{\m}_m[f](t) = \int_0^{\infty} \int_{B-\sigma}^{B+\sigma} F(m,r) w^{\alpha+m+1} r^{-m-2} s(r) f\left(\frac{w}{r},t\right) f(w,t) \, \md r \, \md w \leq 0, \quad \text{for} \quad t \geq 0.
\end{align}
$F$ is defined as \eqref{d:F} and strictly negative for all $r \in [B-\sigma, B+\sigma]$ since $m \in (1,m_*)$. We distinguish the following two cases: 

\begin{description}
	\item[Extinction of all species:] On the one hand, for \emph{$B$ and $\sigma$ not satisfying \eqref{c:Bs}} the functional $\m_m[f](t)$ will decrease in time and its dissipation vanishes when $f$ reaches the trivial steady state \eqref{d:zerostate}. Together with the observation $\m_m[\bar{f}_0]\equiv0$ this implies
	\begin{align*}
		\m_m[f](t) \to 0, \quad \text{for} \quad t \to \infty.
	\end{align*} 
	Hence, $\m_m$ serves as \emph{Lyapunov-type} functional, which can be made rigorous in the setting $\alpha \in (1,m_*)$.
	\begin{lemma}\label{l:extinction}
		Let $f$ be a solution to \eqref{e:PJG} with $\alpha \in (1,m_*)$, $\alpha<3$, feeding preference function chosen as \eqref{d:scomp} and parameters $B, \sigma >0$ such that 
		$$
			0<B-\sigma <1<B+\sigma.
		$$ 
		Further, assume initial conditions \eqref{IC_PJG} satisfying
		$$
			\int_0^\infty w^m f_0 \, \md w < \infty,
		$$
		for $m \in \left(\frac{1+\alpha}{2},\frac{m_*+\alpha}{2}\right) \subset (1,m_*)$ with $m<2$. Then, with $\bar{f}_0$ given by \eqref{d:zerostate},  $f$ satisfies 
		\begin{align}
			\lim_{t \to \infty} f(t) = \bar{f}_0 \qquad \text{in the sense of distributions}.
		\end{align}
	\end{lemma}
	\begin{proof}
		First, we prove that for $m \in \left(\frac{1+\alpha}{2},\frac{m_*+\alpha}{2}\right)$ it has to hold
		\begin{align}\label{c:momenttozero}
			\m_m[f](t) \to 0, \quad \text{for} \quad t \to \infty.
		\end{align} 
		For this, we observe from \eqref{i:decay} that for any $n \in (1,m_*)$ one can estimate
		\begin{align*}
			\dot{\m}_n[f](t) \leq & -C(B, \sigma, n) \int_{\tilde{a}}^{\tilde{b}} \int_{\frac{w}{B+\sigma}}^{\frac{w}{B-\sigma}} w^{\alpha}w'^{n} s\left(\frac{w}{w'}\right) f(w,t)f(w',t) \, \md w' \, \md w \\
			\leq & -C(B, \sigma, n, \tilde{a}, \tilde{b}) \int_{\frac{\tilde{b}}{B+\sigma}}^{\frac{\tilde{a}}{B-\sigma}} \int_{\frac{\tilde{b}}{B+\sigma}}^{\frac{\tilde{a}}{B-\sigma}} w^{\alpha}w'^{n}f(w,t)f(w',t) \, \md w' \, \md w,
		\end{align*}
		given any integration boundaries $\tilde{a} < \tilde{b} \in \R_+$ fulfilling $\tilde{a} < \frac{\tilde{b}}{B+\sigma} < \frac{\tilde{a}}{B-\sigma} < \tilde{b}$ (possible due to the assumption $1 \in (B-\sigma, B+\sigma)$), since under this conditions $s\left(\frac{w}{w'}\right)$ is bounded from below for $(w,w') \in \left[\frac{\tilde{a}}{B-\sigma},\frac{\tilde{b}}{B+\sigma}\right]^2$. By further estimating the prey weight in terms of the predator weight (as in the proofs of Theorem \ref{T:existence1} and Theorem \ref{T:existence2}), before integration of the resulting inequality with respect to $t$, we obtain the following uniform integrability condition in time
		\begin{align}\label{c:unifintt}
			\int_0^t \left(\int_a^b w^{\frac{\alpha+n}{2}}f(w,s) \,\md w \right)^2 \, \md s \leq C(\sigma, B, n, a, b), \quad \forall t >0.
		\end{align}
		This now holds for any 
		\begin{align}\label{c:intb}
			0<a<b \quad \text{close enough such that} \quad a(B+\sigma) >b>a> b (B-\sigma),
		\end{align}
		 which can be seen by setting $a=\frac{\tilde{b}}{B+\sigma}$ and $b=\frac{\tilde{a}}{B-\sigma}$ from before. Thus, for any interval $[a,b]$ away from zero the upper bound \eqref{c:unifintt} holds, since it can always be split it up in sub-intervals where each of their borders fulfil \eqref{c:intb}. Moreover, since for every $m \in \left(1,m_*\right)$ one can find an $\ve >0$ such that $\ve+m<m_*$ and, hence, $\m_{\ve+m}[f](t)$ is bounded uniformly in time due to \eqref{i:decay}, we conclude that the function mapping $w \mapsto w^m f(w,t)$ is a \emph{uniformly tight} sequence w.r.t. time $t$. This allows us to find for every constant $\delta>0$ a time $\tilde{t}>0$ and an integration boundary $L>0$ such that 
		$$
			\int_L^{\infty} w^m f(w,t) \, \md w \leq \delta, \quad \forall t \geq \tilde{t}.
		$$
		 To conclude by a contradicting argument let us assume that
		$$
			\m_m[f](t) \to \m_{\infty} > 0, \text{ as } t \to \infty.
		$$
		From \eqref{i:decay} we immediately obtain that $\m_{\infty} \leq \m_m(t)$ for all $t \geq 0$ has to hold. Due to the uniform tightness argument from before we can find an $L>0$ such that
		\begin{align*}
			\m_m[f](t) \leq \int_0^L w^m f(w,t) \, \md w + \frac{\m_{\infty}}{2}, \quad \forall t \geq \tilde{t}.
		\end{align*}
		We further estimate using the conservation of total biomass \eqref{c:biomasscons} as well as $m>1$
		\begin{align*}
			\int_0^L w^m f(w,t) \, \md w = \int_0^{1/L} w^m f(w,t) \, \md w + \int_{1/L}^L w^m f(w,t) \, \md w \leq L^{1-m}\m + \int_{1/L}^L w^m f(w,t) \, \md w.
		\end{align*}
		This together with the estimate from the uniform tightness yields 
		$$
			\int_{1/L}^L w^m f(w,t) \, \md w \geq \frac{\m_{\infty}}{2} - L^{1-m}\m, \quad \forall t \geq \tilde{t},
		$$
		where under the assumption $\m_{\infty}>0$ the right-hand-side can be made positive by choosing $L$ big enough. This immediately contradicts the uniform bound \eqref{c:unifintt} by setting $m=\frac{\alpha+n}{2}$. Indeed, for $m \in \left(\frac{1+\alpha}{2},\frac{m_*+\alpha}{2}\right)$ and $\alpha \in (1,m_*)$ a straightforward calculation shows that we have $n=2m-\alpha \in (1,m_*)$, for which estimate \eqref{c:unifintt} holds. 
		Since the limit $\m_{\infty}$ is non-negative, we conclude $\m_m[f](t) \to 0$ as $t \to \infty$ along every solution $f$ to \eqref{e:PJG}-\eqref{IC_PJG}.
		
		Second, we observe that for any test-function $\vp \in \mathcal{C}_0^\infty([0,\infty))$ one can estimate using the mean value theorem
		\begin{align*}
			&\left|\int_0^\infty w f(w,t) \vp(w) \, \md w -\m \vp(0)\right| =\left|\int_0^\infty w f(w,t) \left(\vp(w)-\vp(0)\right) \, \md w\right|  \\
			&\leq  \int_0^\infty w^m f(w,t) w^{1-m}\left|\vp(w)-\vp(0)\right| \, \md w \leq \m_m(t) \sup_{w \in [0,\infty), \tilde{w} \in [0,w]} \left|w^{2-m}\vp'(\tilde{w})\right|.
		\end{align*}
		The result follows, from \eqref{c:momenttozero} and since the condition $\alpha < 3$ allows to choose $m$ close enough to one such that $2-m>0$, which together with $\vp \in \mathcal{C}_0^\infty([0,\infty))$ ensures that the second factor is bounded. 
	\end{proof}
	

	\item[Cascade effect:]\label{d:cascade} On the other hand, if the parameters \emph{$B,\sigma>0$ of the feeding kernel are compatible with condition \eqref{c:Bs}}, then the $m$-th moment $\m_m[f](t)$ is strictly decreasing in time along a solution $f$ of \eqref{e:PJG}-\eqref{IC_PJG} until $f$ satisfies the sufficient equilibrium condition \eqref{c:gaps}. Due to the lack of a concrete form of non-trivial equilibria in case \eqref{c:Bs} and the locality of the aforementioned condition, a rigorous convergence result of $f$ to such a non-trivial steady state fulfilling \eqref{c:gaps} is not possible and will be matter of further investigations. However, this leaves us with a strong indication that in such regime the aquatic system will converge towards a non-zero stationary size-distribution with gaps in its spectrum. 

Numerical simulations in Section \ref{s:num} give evidence to this heuristic argument. See also Figure \ref{f:cascades} representing the support of such a non-trivial steady state: The solid black lines encode the intervals in the size spectrum, in which extinction happened, while the intervals coloured in purple determine the segregated trophic levels. Important to notice here is that such non-trivial steady states for this model are caused by the underlying mechanics of predation: In such a setting individuals no longer feed and are no longer being fed upon, hence the biomass is stuck and immobile. In the natural setting, it is expected that they just describe quasi-steady states, i.e. states towards which the size-distribution converges fast due to the strong effect of predation within the ecosystem, but as soon as they are reached, different dynamics within the ecological environment (e.g., mutualism, starvation, natural mortality, change of predation habits) might force the graph of the size-distribution to converge to a different form.
\end{description}

\subsection{Power-law Equilibria}\label{ss:powerlaw}
In huge marine ecosystems it is well observed \cite{BD,SPS, SSP} that in logarithmic intervals of the organism body size the biomass is approximately constant. This being equivalent to the biomass in logarithmic scales, having slope -1, while transformed into the original variables it translates to a \emph{power-law} with exponent -2. Ecological evidence, on the other hand, provides that in local ecosystems size-distributions are rather different from a power-law, as the aforementioned cascade effect (see Section \ref{d:cascade}), are observed \cite{Eetal,HSS,RGK}. The finding of the aforementioned power-law size-spectrum distribution is based on investigating a large quantity of data from huge ecosystems \cite{SPS}, which naturally combine a variety of feeding behaviours in different food-webs. Hence, we expect that such a power-law phenomenon is difficult to capture with existing size-spectrum models. 

Indeed, in \cite{BR, DDL} it was proven that for their models such a power law stationary solution $w^{\gamma}$ exists with exponent $\gamma$ close to -2 in an appropriate parameter-regime. The local stability analysis of the power-law equilibrium, performed in the follow-up work \cite{DDLP} with methods from spectral analysis, showed that stability of the power-law state is very unlikely. It could only be proved in a very specific parameter-regime being characterised by a very low predator/prey mass ratio $B$, a unusual high assimilation efficiency $K$ and a relatively large diet breadth $\sigma$. Moreover, it should be mentioned that such power-law solutions or linear combinations of such are known to appear as equilibria solutions to the structurally very similar coagulation-fragmentation equations, see \cite{DGS} for instance.

Also equation \eqref{e:PJG} admits such a power-law steady state, which can be seen by inserting the ansatz $\bar{f}(w)=w^\gamma$ into the stationary equation \eqref{e:stat}. We obtain the following condition for the exponent $\gamma$:
\begin{align*}
	0=&r^{\alpha+\gamma}(r+K)^{-\alpha-2-2\gamma} + (1-K)K'^{-3-\alpha-2\gamma} r^{\alpha + \gamma} -r^{-2-\gamma}- r^{\alpha+\gamma} \\
	=&r^{\alpha + \gamma} G(\gamma,r),
\end{align*}
where we defined
$$
	G(\gamma,r):=\left((r+K)^{-\alpha-2-2\gamma} + (1-K)K'^{-3-\alpha-2\gamma} -r^{-2-2\gamma-\alpha}- 1\right).
$$
It is easily seen that $G(\gamma,r)=0$ for all $r \in\supp{(s)}$ when
$$
\gamma:=-\frac{\alpha+3}{2}.
$$
We further observe that $\gamma>-2$ for $\alpha <1$, $\gamma<-2$ for $\alpha >1$ and $\gamma = 2$ for $\alpha=1$.  
\begin{rem}
	Coherent with our expectation of not finding such power-law distributions for our model, we observe that the state
	$$
		\bar{f}(w):=w^{-\frac{\alpha+3}{2}},
	$$ 
	cannot be a feasible solution of \eqref{e:PJG}-\eqref{IC_PJG}, since the total biomass \eqref{c:biomasscons} is infinite once the system reaches this state. 
\end{rem}

\section{Numerical Simulations}\label{s:num}

To illustrate some of the findings of this article, we present numerical simulations of equation \eqref{eq:e:BEpdelta}.

Specifically, we chose a fixed bound $W$, big enough in relation to the other constants of the system, and a finite time $T$ and simulated the equation on $[0, W] \times [0, T]$. The numerical scheme in which this equation is implemented, takes use of a semi-discretisation of the equation, namely a discretisation in the variable $w$:
$$
	\frac{d f_n^N}{dt} = Q^N\left(f^N,f^N\right)_n, \quad n \in \{0,\dots,N\}.
$$
We chose an integer $N$ and divided the interval $[0, W]$ in $N$ intervals of equal size, producing the equidistant grid with grid-points 
$$
	w_n:=n\frac{W}{N}, \quad n \in \{0,\dots,N\}.
$$ 
The numerical unknown is a vector
$$
	f^N(t):=(f_0(t), \dots, f_N(t)) \approx (f(w_0,t),\dots, f(w_N,t)),
$$
while $Q^N$ represents the integrals on the right-hand side of \eqref{eq:e:BEpdelta}, which are approximated by the \emph{trapezoidal rule}. When needed, the value of $f$ at a given point not on the grid is approximated through linear interpolation of its two closest neighbours, i.e. for $w \notin \{w_0,\dots,w_N\}$ we approximate
$$
	f(w,t)\approx \hat{f}(w,t) := \frac{f_{l-1}(t)(w_l-w)+f_l(t)(w-w_{l-1})}{w_l-w_{l-1}}, \quad \text{with }l\text{ s.t. } \,\, w \in [w_{l-1},w_l].
$$ 
This leads us to 
\begin{align*}
	Q^N\Big(f^N,f^N\Big)_n:=& \frac{W}{2N} \sum_{l:\, w_l < \frac{w_n}{K}} \Big[k(w_n-Kw_{l-1},w_{l-1}) \hat{f}(w_n-Kw_{l-1})f_{l-1} +  k(w_n-Kw_l,w_l) \hat{f}(w_n-Kw_l)f_l\Big] \\
	+&\frac{W}{2N} \frac{1-K}{K'^2} \hat{f}\left(w_n/K'\right) \sum_{l=1}^N \Big[k\left(w_{l-1}, w_n/K'\right)f_{l-1} + k\left(w_l, w_n/K'\right) f_l \Big]\\
	-& \frac{W}{2N} f_n \sum_{l=1}^N \Big[\left(k(w_n,w_{l-1})+k(w_{l-1},w_n)\right) f_{l-1} + \left(k(w_n,w_l)+k(w_l,w_n)\right) f_l \Big]\,.
\end{align*}
For the time-discretisation a time-adaptive Runge-Kutta scheme was used. The scheme was implemented in \textsc{Python}, where we used \textsc{integrate.solve\_ivp} from the \textsc{SciPy}-package. Moreover, numerical experiments were exclusively performed for the choice \eqref{g:kernel} for the feeding kernel $k(\cdot,\cdot)$ both with feeding preference function $s(\cdot)$ given by \eqref{d:sgauss} and \eqref{d:scomp}.

\subsection{Numerical Simulations for the Feeding Preference Function with Compact Support}

Simulations are carried out with $N=200$, fixed upper-size bound $W=10$ and final time $T=5$. The initial distribution is chosen as a linear interpolation between $f_0=10$ and $f_N=0.1$. In the figures the solution is plotted at different time-steps of the simulation: at starting with the initial distribution in the up-left subfigure, and ending with the final distribution at time $T$ in the bottom-right figure. 

In Figure \ref{SCa09} the simulation was performed for $\alpha=0.9$, $B=1.5$ and $\sigma=0.3$, $K=0.1$ and $K'=0.01$, hence a parameter-regime allowing a non-trivial equilibrium \eqref{c:gaps}. Although an analytical proof is not provided in Section \ref{ss:longtime} the simulations show convergence to a solution with gaps in the size-spectrum, which remains constant after time $T=5$. It can be seen clearly that they become smaller giving evidence to the calculations in the previous Section \ref{s:ss}, \eqref{c:gaps}. Additionally, it should be mentioned that the upper bound $W=w_0=\max\{\supp(f_0)\}=10$ does not serve as reference value $\bar{w}$ for the body-size. Indeed, at the stationary state a small interval left from $w_0$, where the solution takes value zero, can be observed. This is a consequence of the limits of our simulations, which does not show growth of an individual beyond the upper bound $W$.

In Figure \ref{SCa11} the same effect can be observed for $\alpha=1.1$, with the other parameters chosen the same as in Figure \ref{SCa09}, emphasising that for our model the choice of $\alpha $ does not have significant impact on the qualitative behaviour of the solution.

With parameter values $B=1.1$ and $\sigma=0.3$ condition \eqref{c:Bs} is violated, hence extinction of the ecosystem is to be expected due to Lemma \ref{l:extinction}, which is visualised in Figure \ref{SCa11B11}.

In Figure \ref{SCa09B15K04} the simulation was performed for $\alpha=0.9$, $B=1.5$ and $\sigma=0.3$, $K=0.4$ and $K'=0.01$, again allowing a non-trivial equilibrium with gaps in the size-spectrum. Different to the simulations in Figure \ref{SCa09} one can observe domes with much smaller altitude, which is related to the higher assimilation efficiency $K$ enhancing the growth of the predating organisms and hence enforcing a stronger drift to the right. 

\begin{figure}[ht]
\centering
\begin{subfigure}{.45\textwidth}
    \centering
    \includegraphics[width=.85\linewidth]{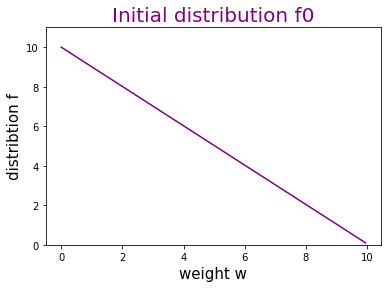}  
\end{subfigure}
\hspace{-0.5cm}
\begin{subfigure}{.45\textwidth}
    \centering
    \includegraphics[width=.85\linewidth]{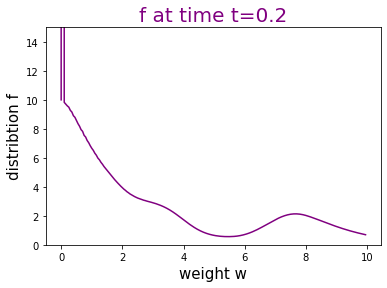}  
\end{subfigure}

\begin{subfigure}{.45\textwidth}
    \centering
    \includegraphics[width=.85\linewidth]{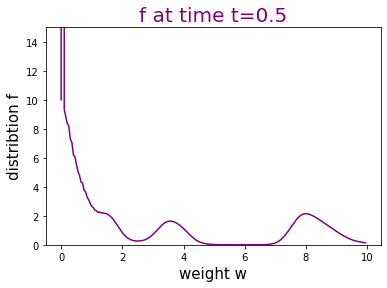}  
\end{subfigure}
\hspace{-0.5cm}
\begin{subfigure}{.45\textwidth}
    \centering
    \includegraphics[width=.85\linewidth]{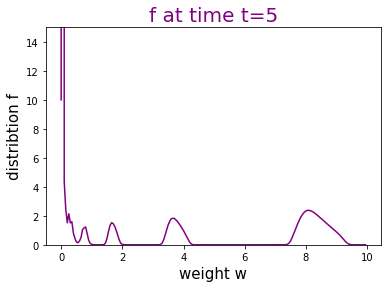}  
\end{subfigure}
\caption{Simulations with compactly supported feeding preference function starting from linear initial conditions with parameters $\alpha=0.9$, $B=1.5$, $\sigma=0.3$, $K=0.1$ and $K'=0.01$, showing convergence to a steady-state representing the \emph{cascade-effect}.}
\label{SCa09}
\end{figure}

\begin{figure}[H]
\centering
\begin{subfigure}{.45\textwidth}
    \centering
    \includegraphics[width=.85\linewidth]{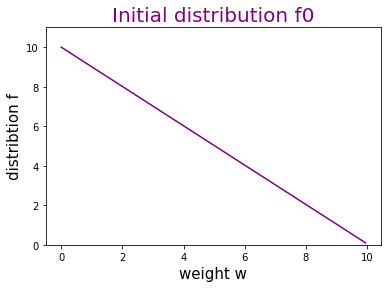}  
\end{subfigure}
\hspace{-0.5cm}
\begin{subfigure}{.45\textwidth}
    \centering
    \includegraphics[width=.85\linewidth]{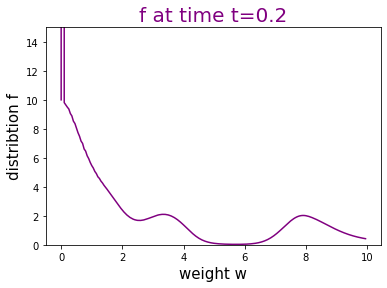}  
\end{subfigure}

\begin{subfigure}{.45\textwidth}
    \centering
    \includegraphics[width=.85\linewidth]{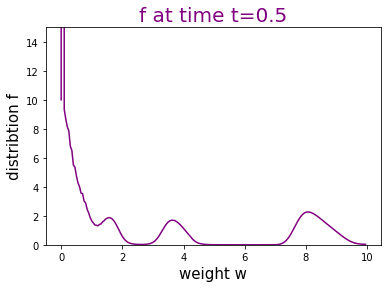}  
\end{subfigure}
\hspace{-0.5cm}
\begin{subfigure}{.45\textwidth}
    \centering
    \includegraphics[width=.85\linewidth]{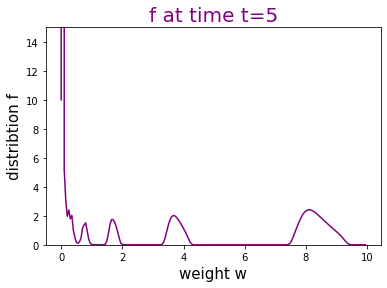}  
\end{subfigure}
\caption{Simulations with compactly supported feeding preference function starting from linear initial conditions with parameters $\alpha=1.1$, $B=1.5$, $\sigma=0.3$, $K=0.1$ and $K'=0.01$, showing convergence to a steady-state representing the \emph{cascade-effect}.}
\label{SCa11}
\end{figure}

\begin{figure}[ht]
\centering
\begin{subfigure}{.45\textwidth}
    \centering
    \includegraphics[width=.85\linewidth]{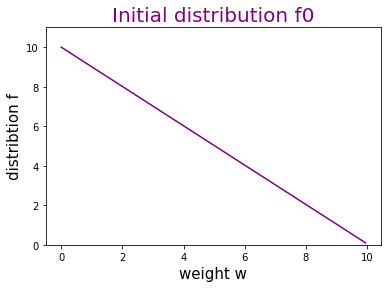}  
\end{subfigure}
\hspace{-0.5cm}
\begin{subfigure}{.45\textwidth}
    \centering
    \includegraphics[width=.85\linewidth]{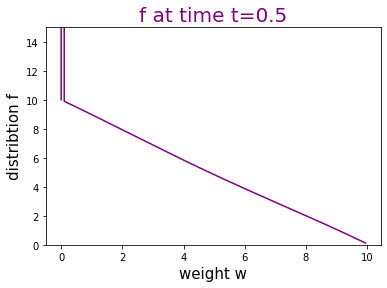}  
\end{subfigure}

\begin{subfigure}{.45\textwidth}
    \centering
    \includegraphics[width=.85\linewidth]{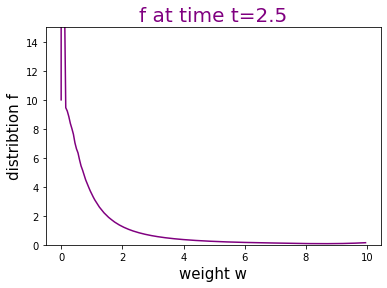}  
\end{subfigure}
\hspace{-0.5cm}
\begin{subfigure}{.45\textwidth}
    \centering
    \includegraphics[width=.85\linewidth]{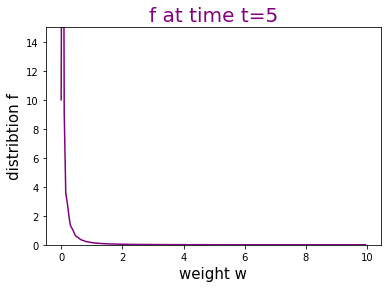}  
\end{subfigure}
\caption{Simulations with compactly supported feeding preference function starting from linear initial conditions with parameters $\alpha=1.1$, $B=1.1$, $\sigma=0.3$, $K=0.1$ and $K'=0.01$, showing convergence to the trivial steady-state representing the \emph{extinction of all species}.}
\label{SCa11B11}
\end{figure}

\begin{figure}[ht]
\centering
\begin{subfigure}{.45\textwidth}
    \centering
    \includegraphics[width=.85\linewidth]{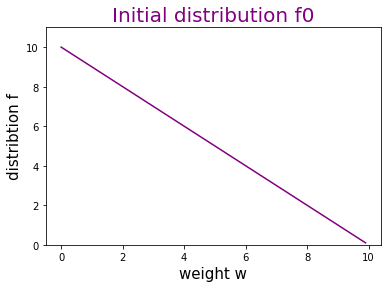}  
\end{subfigure}
\hspace{-0.5cm}
\begin{subfigure}{.45\textwidth}
    \centering
    \includegraphics[width=.85\linewidth]{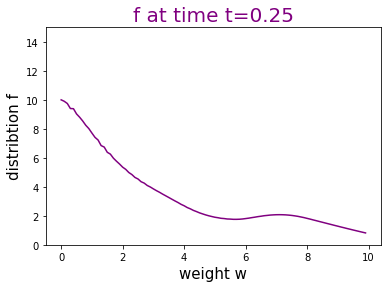}  
\end{subfigure}

\begin{subfigure}{.45\textwidth}
    \centering
    \includegraphics[width=.85\linewidth]{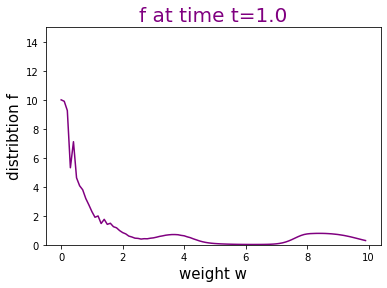}  
\end{subfigure}
\hspace{-0.5cm}
\begin{subfigure}{.45\textwidth}
    \centering
    \includegraphics[width=.85\linewidth]{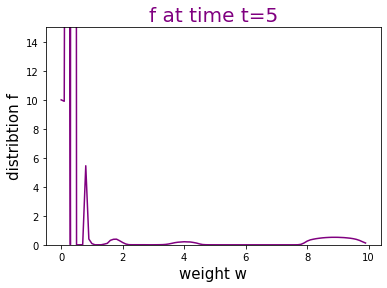}  
\end{subfigure}
\caption{Simulations with compactly supported feeding preference function starting from linear initial conditions with parameters $\alpha=0.9$, $B=1.5$, $\sigma=0.3$, $K=0.4$ and $K'=0.01$, showing convergence to a steady-state representing the \emph{cascade-effect} with flat domes.}
\label{SCa09B15K04}
\end{figure}

\subsection{Numerical Simulations for the Gaußian Feeding Preference Function}

Although many of the analytical results for the model \eqref{e:PJG} with the Gaußian feeding preference function \eqref{d:sgauss} are not valid, with numerical simulations we are able to indicate some interesting behaviours.

In Figure \ref{SGa09B15} $\alpha=0.9$, $B=1.5$ and $\sigma=0.3$, $K=0.1$ and $K'=0.01$, hence the same parameters as for the simulation with the compactly supported feeding preference function in Figure \ref{SCa09}, where chosen. Unlike Figure \ref{SCa09}, convergence to the trivial steady can be observed. Although once observes that the solution first forms domes suggesting the convergence to the non-trivial steady state with gaps in the size-spectrum, before all mass is absorbed at 0. 

Condition \eqref{c:Bs} does not seem to be sufficient for the Gaußian feeding preference function. Convergence to a non-trivial equilibrium can be achieved also for the Gaußian case by choosing $B$ large enough and the variance $\sigma$ small enough, as it can be seen in Figure \ref{SGa09B2sig02} with $B=2$ and $\sigma=0.2$.

Clarifying the equilibrium conditions as well as a possible regime for metastability of the non-trivial equilibrium, suggested, e.g., by Figure \ref{SGa09B15}, for the Gaußian feeding preference function will be subject of further investigations.

\begin{figure}[ht]
\centering
\begin{subfigure}{.45\textwidth}
    \centering
    \includegraphics[width=.85\linewidth]{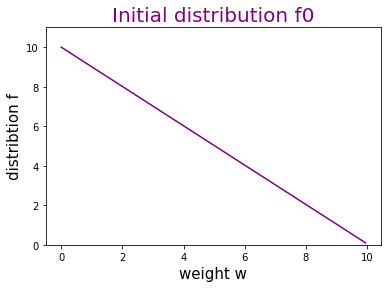}  
\end{subfigure}
\hspace{-0.5cm}
\begin{subfigure}{.45\textwidth}
    \centering
    \includegraphics[width=.85\linewidth]{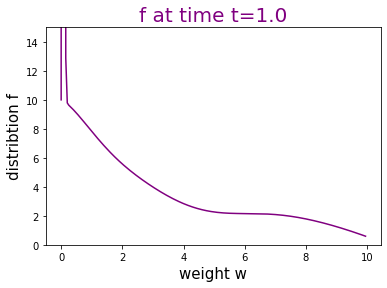}  
\end{subfigure}

\begin{subfigure}{.45\textwidth}
    \centering
    \includegraphics[width=.85\linewidth]{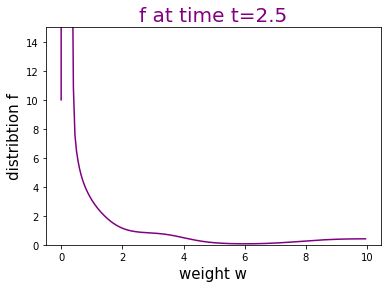}  
\end{subfigure}
\hspace{-0.5cm}
\begin{subfigure}{.45\textwidth}
    \centering
    \includegraphics[width=.85\linewidth]{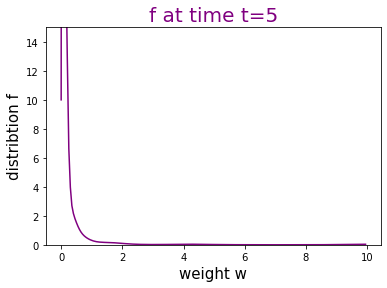}  
\end{subfigure}
\caption{Simulations with Gaußian feeding preference function starting from linear initial conditions with parameters $\alpha=0.9$, $B=1.5$, $\sigma=0.3$, $K=0.1$ and $K'=0.01$, showing convergence to the trivial steady-state representing the \emph{extinction of all species}.}
\label{SGa09B15}
\end{figure}

\begin{figure}[ht]
\centering
\begin{subfigure}{.45\textwidth}
    \centering
    \includegraphics[width=.85\linewidth]{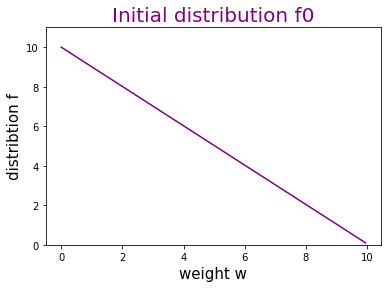}  
\end{subfigure}
\hspace{-0.5cm}
\begin{subfigure}{.45\textwidth}
    \centering
    \includegraphics[width=.85\linewidth]{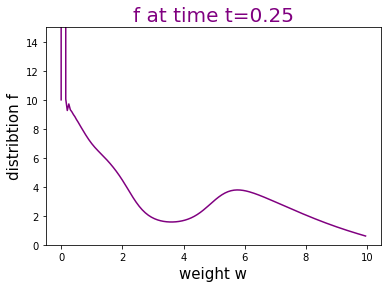}  
\end{subfigure}

\begin{subfigure}{.45\textwidth}
    \centering
    \includegraphics[width=.85\linewidth]{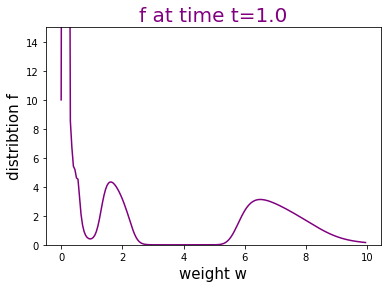}  
\end{subfigure}
\hspace{-0.5cm}
\begin{subfigure}{.45\textwidth}
    \centering
    \includegraphics[width=.85\linewidth]{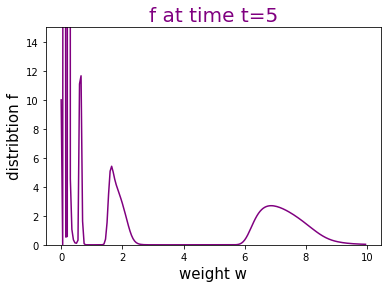}  
\end{subfigure}
\caption{Simulations with Gaußian feeding preference function starting from linear initial conditions with parameters $\alpha=0.9$, $B=2$, $\sigma=0.2$, $K=0.1$ and $K'=0.01$, showing convergence to a steady-state representing the \emph{cascade-effect}.}
\label{SGa09B2sig02}
\end{figure}

\section{Conclusion and Outlook}\label{s:concl}

In order to investigate the time-evolution of the size-spectrum within an aquatic ecosystem, we proposed a model governed by a non-local integral equation with binary interaction describing predation events between two organisms in the ecosystem. Our proposed model extends the similar one introduced in \cite{DDL} by a term ensuring the \emph{conservation of biomass} within the ecosystem via creation of a certain amount of very small organisms at every predation event, which ensures re-utilisation of all materials within the ecosystem. While the model in its full generality admits that the size of the 'offspring' is drawn by a probability distribution, which depends on the prey size while having very small mean, we restricted our further considerations to the case where the size of the small individuals is deterministically given as a small fraction of the prey size. Our proposed equation provides a model for the dynamics within a closed or almost closed ecosystem, whose fragility can be seen due to the high impact outer influence has.

Analytical investigation of this deterministic model revealed several insights of the behaviour, confirmed and further investigated with numerical experiences. Most results are valid for the feeding preference function with compact support \eqref{d:scomp}, since this allows us to partly localize the integral operators defining the right-hand-side of \eqref{e:PJG}. Indeed, for such a compact feeding preference  $s(\cdot)$ we could show the existence of an interval $(1,m_*)$, such that for $m \in (1,m_*)$ the $m$-th moment of the distribution is non-increasing in time, which in a next step ensured global in time existence of solutions in the $w$-weighted $L^1$-space if $\alpha \in [1,m_*]$. For $\alpha < 1$ we provide a local-in time existence result. Although global in time existence is limited to a parameter regime with search-volume exponent $\alpha \in (1,m_*)$ while literature suggests an $\alpha$ lower than 1 \cite{W}, numerical simulations indicated that the choice of $\alpha$, as long as sufficiently chosen of order 1, does not have significant impact on the qualitative behaviour of the solution (compare Figures \ref{SCa09} and \ref{SCa11}). 

Moreover, analytical investigations revealed the occurrence of a trivial steady state, representing a \emph{died-out ecosystem} and given by a distribution having all the mass concentrated at $w=0$. Convergence to this steady state could be shown for the case where the feeding preference function \eqref{d:scomp} with compact support allows predation on organisms of the same size as the predating individual, which is supported by numerical tests (see Figure \ref{SCa11B11}). Necessary conditions for existence of a non-trivial steady state representing the aforementioned \emph{cascade effect} could be derived for the case of the compactly supported feeding-preference function and convergence of the size-distribution function to such \emph{dome patterns} could be observed in numerical experiments, as long as predation on species with the same size is not allowed (see Figure \ref{SCa09}). Although such cascade effects are widely observed in nature, it usually describes a temporary steady state. In the contrary, this theoretical setup forces the size-spectrum of the ecosystem to stay unchanged due to the lack of predation possibilities within the species, while lack of starvation and natural death rate on the remaining individuals means abundances stay constant rather than decrease.

Numerical tests suggest that most of the analytical results for the feeding preference function with compact support are also valid for the Gaußian feeding preference function, although just in a restricted parameter-regime. For large enough preferred feeding ratio and small variance convergence to dome patterns can be observed (Figure \ref{SGa09B2sig02}), while convergence to the trivial steady state also occurs in the Gaussian case (Figure \ref{SGa09B15}), although for the same set of parameters we observe the cascade-effect for the compactly supported feeding-preference function. In addition to this, we investigated the possibility of a power-law steady state and indeed found that equation \eqref{e:PJG} admits one of such, with power $-\frac{3+\alpha}{2}$, which fits to previous investigations \cite{BR, DDL, DDLP, SPS}. Since due to its violation of the conservation law it is not a feasible equilibrium of our model, which is coherent with the observation that such power-laws are usually found after processing large quantity of data from huge ecosystems.

The findings within this article naturally paved the way to perform further investigations both from an analytical side as well as from an observational point of view. For the latter, fully identifying the asymptotic behaviour for the Gaußian feeding preference function, much used and studied in existing literature \cite{BR, DDL, DDLP}, is of high interest. Furthermore, a task of severe importance is to perform a full characterisation of the dome-patterns emerging as a non-trivial steady state followed by a stability analysis of the model presented here in dependence of the model parameters. Indeed, while we were already able to answer the ecologically relevant question of the location and size of the gaps, to characterise the shapes of the domes in dependence of the parameters remained an open problem, while comparing Figure \ref{SCa09} and \ref{SCa09B15K04} already suggested that their hight and width are depending on the \emph{assimilation efficiency}. Comparing our findings to - as well as calibrating the model with - data could help in a next step to predict the influence outer impacts have on an ecosystem. 

\newpage
\vspace{1cm}
\paragraph{Acknowledgements:} We thank \emph{Nicolas Loeuille} and \emph{Vincent Calvez} for fruitful discussions, providing helpful insights and broadening our view on the existing literature. We also thank the careful referees for pointing out several improvements of a first version.

 L.K. received funding by a grant from the FORMAL team at ISCD - Sorbonne Université and by the European Commission under the Horizon2020 research and innovation programme, Marie Sklodowska-Curie grant agreement No 101034255. B.P. has received funding from the European Research Council (ERC) under the European Union Horizon 2020 research and innovation programme (grant agreement No 740623).
 
\includegraphics[width=.1\linewidth, right]{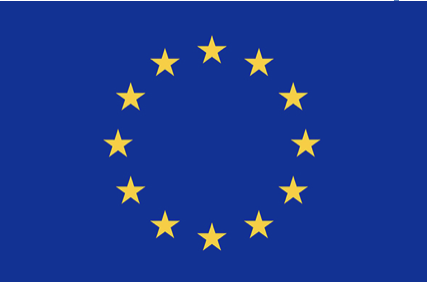}
\paragraph{Declarations Competing Interests:} The authors certify that they do not have any conflict of interest.
\vspace{-12pt}
\paragraph{Data availability :} Does not apply.
\vspace{-12pt}
\paragraph{Code availability :} On request.


\end{document}